\documentclass[a4paper,12pt]{amsart}
\usepackage{amssymb, amsmath, amsthm, mathtools, mathrsfs, bm}
\usepackage{amsrefs}%amsplain with []
\usepackage{tikz-cd}
\usepackage{extarrows}
\usepackage{arydshln} % draw horizontal/vertical dash-lines in array and tabular environments
\usepackage{enumerate}
\usepackage{booktabs,float}
\usepackage{bbm}
\usepackage[colorlinks = true,
            linkcolor = blue,
            urlcolor  = blue,
            citecolor = teal,
            anchorcolor = blue]{hyperref}

\theoremstyle{plain}
\newtheorem{theorem}{Theorem}[section]
\newtheorem{proposition}[theorem]{Proposition}
\newtheorem{corollary}[theorem]{Corollary}
\newtheorem{lemma}[theorem]{Lemma}

\theoremstyle{definition}

\theoremstyle{remark}
\newtheorem{remark}[theorem]{Remark}

\newcommand{\xra}{\xlongrightarrow}

\newcommand{\CP}[1]{\mathbb{C}P^{#1}}
\newcommand{\HP}[1]{\mathbb{H}P^{#1}}

\newcommand{\id}{\mathbbm{1}}

\newcommand{\an}[1]{\ensuremath{\mathbf{A}_{n}^{#1}}}

\newcommand{\PD}{Poincar\'{e} duality\,}

\newcommand{\Z}{\ensuremath{\mathbb{Z}}}

\newcommand{\z}[1]{\ensuremath{\mathbb{Z}/2^{#1}}}
\newcommand{\zp}[1]{\ensuremath{\mathbb{Z}/p^{#1}}}
\newcommand{\wk}{\widetilde{K}}
\newcommand{\ko}{\widetilde{KO}}

\DeclareMathOperator{\coker}{coker}
\DeclareMathOperator{\im}{im}
\DeclareMathOperator{\Sq}{Sq}

\DeclareMathOperator{\Wh}{Wh}

\newcommand{\mat}[4]{\ensuremath{\ensuremath{\left[\begin{smallmatrix}
  #1&#2\\
  #3&#4
\end{smallmatrix}\right]}}}

\newcommand{\matwo}[2]{\ensuremath{\footnotesize{\begin{bmatrix}
  #1\\
  #2
\end{bmatrix}}}}

\newcommand{\smatwo}[2]{\ensuremath{\left[\begin{smallmatrix}
 #1\\
 #2
\end{smallmatrix}\right]}}

\title[Suspension Homotopy of $5$-manifolds]{The Homotopy Decomposition of the Suspension of a non-simply-connected $5$-manifold}

\author[P. Li]{Pengcheng Li}
\address{Department of Mathematics, School of Sciences, Great Bay University, Dongguan, Guangdong \rm{523000}, China}
\email{lipcaty@outlook.com}
\urladdr{https://lipacty.github.io}

\author[Z. Zhu]{Zhongjian Zhu}
\address{College of Mathematics and Physics, Wenzhou University, Wenzhou, Zhejiang \rm{325035}, China}
\email{zhuzhongjian@amss.ac.cn}


\begin{document}

\begin{abstract}
 In this paper we determine the homotopy types of the reduced suspension space of certain connected orientable closed smooth $5$-manifolds. As applications, we compute the reduced $K$-groups of $M$ and show that the suspension map between the third cohomotopy set $\pi^3(M)$ and the fourth cohomotopy set $\pi^4(\Sigma M)$ is a bijection.
\end{abstract}
\maketitle

\section{Introduction}\label{sec:intro}

One of the goals of algebraic topology of manifolds is to determine the homotopy type of the (reduced) suspension space $\Sigma M$ of a given manifold M. This problem has attracted a lot of attention since So and Theriault's work \cite{ST19}, which showed how the homotopy decompositions of the (double) suspension spaces of manifolds can be used to characterize some important invariants in geometry and mathematical physics, such as reduced $K$-groups and gauge groups. Several works have followed this direction, such as \cite{Huang-6mflds,Huang21,Huang22,HL,lipc23,CS22}. The integral homology groups $H_\ast(M)$ serve as the fundamental input for this topic. As shown by these papers, the 2-torsion of $H_\ast(M)$ and potential obstructions from certain Whitehead products usually prevent a complete  homotopy classification of the (double) suspension space of a given manifold $M$.

The main purpose of this paper is to investigate the homotopy types of the suspension of a non-simply-connected orientable closed smooth $5$-manifold. Notice that Huang \cite{Huang21} studied the suspension homotopy of $5$-manifolds $M$ that are $S^1$-principal bundles over a simply-connected oriented closed $4$-manifold. The homotopy decompositions of $\Sigma^2M$ are successfully applied to determine the homotopy types of the pointed looped spaces of the gauge groups of a principal bundle over $M$. In this paper we greatly loosen the restriction on the homology groups $H_\ast(M)$ of the non-simply-connected $5$-manifold $M$ by assuming that $H_1(M)$ has a torsion subgroup that is not divided by $6$ and $H_2(M)$ contains a general torsion part. 

To state our main results, we need the following notion and notations.
Let $n\geq 2$. Denote by $\eta=\eta_n=\Sigma^{n-2}\eta$ the iterated suspension of the first Hopf map $\eta\colon S^3\to S^2$. 
Recall from (cf. \cite{Todabook}) that $\pi_3(S^2)\cong\Z\langle \eta\rangle$, $\pi_{n+1}(S^n)\cong \z{}\langle \eta \rangle$ for $n\geq 3$ and $\pi_{n+2}(S^n)\cong \z{}\langle \eta^2 \rangle$.
For an abelian group $G$, denote by $P^{n+1}(G)$ the Peterson space characterized by having a unique reduced cohomology group $G$ in dimension $n+1$; in particular, denote by $P^{n+1}(k)=P^{n+1}(\Z/k)$ the mod $k$ Moore space of dimension $n+1$, where $\Z/k$ is the group of integers modulo $k$, $k\geq 2$. There is a canonical homotopy cofibration 
\[S^{n}\xra{k}S^{n}\xra{i_{n}}P^{n+1}(k)\xra{q_{n+1}}S^{n+1},\]
where $i_n$ is the inclusion of the bottom cell and $q_{n+1}$ is the pinch map to the top cell. 
Recall that for each prime $p$ and integer $r\geq 1$, there are  higher order Bockstein operations $\beta_r$ that detect the degree $2^r$ map on spheres $S^n$. For each $r\geq 1$, there are canonical maps $\tilde{\eta}_r\colon S^{n+2}\to P^{n+2}(2^r)$ satisfying the relation $q_{n+1}\tilde{\eta}_r=\eta$, see Lemma \ref{lem:Moore1}. 
A finite CW-complex $X$ is called an \emph{$\an{2}$-complex} if it is $(n-1)$-connected and has dimension at most $n+2$. In 1950, Chang \cite{Chang50} proved that for $n\geq 3$, every $\an{2}$-complex $X$ is homotopy equivalent to a wedge sum of finitely many spheres and mod $p^r$ Moore spaces with $p$ any primes and the following four \emph{elementary (or indecomposable) Chang complexes}: 
\begin{align*}
  C^{n+2}_\eta&=S^n\cup_\eta\bm C S^{n+1}=\Sigma^{n-2}\CP{2},~~ 
  C^{n+2}_r=P^{n+1}(2^r)\cup_{i_{n}\eta}\bm C S^{n+1},\\
  C^{n+2,s}&=S^n\cup_{\eta q_{n+1}}\bm C P^{n+1}(2^s),~~
  C^{n+2,s}_r=P^{n+1}(2^r)\cup_{i_n\eta q_{n+1}}\bm C P^{n+1}(2^s),
\end{align*} 
where $\bm{C}X$ denotes the reduced cone on $X$ and $r,s$ are positive integers.  We recommend \cite{ZP17, ZP21, ZLP19,ZP23,lipc2-an2} for recent work on the homotopy theory of Chang complexes.

Now it is prepared to state our main result.
Let $M$ be an orientable closed $5$-manifold whose integral homology groups are given by
\begin{equation}\label{HM}
  \begin{tabular}{cccccccc}
    \toprule 
    $i$&$1$&$2$&$3$&$4$&$0,5$&$\geq 6$\\
    \midrule 
    $H_i(M)$&$\Z^l\oplus H$&$\Z^d\oplus T$&$\Z^d\oplus H$&$\Z^l$&$\Z$&$0$
   \\ \bottomrule
  \end{tabular},
\end{equation}
where $l,d$ are positive integers and $H,T$ are finitely generated torsion abelian groups. 

\begin{theorem}\label{thm:main}
  Let $M$ be an orientable smooth closed $5$-manifold with $H_\ast(M)$ given by (\ref{HM}). Let $T_2\cong \bigoplus_{j=1}^{t_2}\z{r_j}$ be the $2$-primary component of $T$ and suppose that $H$ contains no $2$- or $3$-torsion. There exist integers $c_1,c_2$ that depend on $M$ and satisfy
  \[0\leq c_1 \leq \min\{l,d\},\quad 0\leq c_2\leq \min\{l-c_1,t_2\}\] 
  and $c_1=c_2=0$ if and only if the Steenrod square $\Sq^2$ acts trivially on $H^2(M;\z{})$. Denote $T[c_2]=T/\oplus_{j=1}^{c_2}\z{r_j}$.
\begin{enumerate}[1.]
  \item\label{thm-spin} Suppose $M$ is spin, then there is a homotopy equivalence 
  \begin{multline*}
    \Sigma M\simeq \big(\bigvee_{i=1}^lS^2\big)\!\vee\! \big(\bigvee_{i=1}^{d-c_1}S^3\big)\!\vee\! \big(\bigvee_{i=1}^{d}S^4\big)\!\vee\! \big(\bigvee_{i=1}^{l-c_1-c_2}S^5\big)\!\vee\! P^3(H)\!\vee\! P^5(H)\\\!\vee\! \big(\bigvee_{i=1}^{c_1} C^5_\eta\big)\!\vee\! P^4(T[c_2])\!\vee\!\big(\bigvee_{j=1}^{c_2}C^5_{r_j}\big)\!\vee\! S^6.
  \end{multline*}

\item\label{thm-nonspin} Suppose $M$ is non-spin, then there are three possibilities for the homotopy types of $\Sigma M$.
\begin{enumerate}
  \item\label{nonspin:a} If for any $u, v\in H^4(\Sigma M;\z{})$ satisfying $\Sq^2(u)\neq 0$ and $\Sq^2(v)=0$, there holds $u+v\notin \im(\beta_r)$ for any $r\geq 1$,
then there is a homotopy equivalence 
\begin{multline*}
  \Sigma M\simeq \big(\bigvee_{i=1}^lS^2\big)\!\vee\! \big(\bigvee_{i=1}^{d-c_1}S^3\big)\!\vee\! \big(\bigvee_{i=2}^{d}S^4\big)\!\vee\! \big(\bigvee_{i=1}^{l-c_1-c_2}S^5\big)\!\vee\! P^3(H)\!\vee\! P^5(H)\\\!\vee\! \big(\bigvee_{i=1}^{c_1} C^5_\eta\big)\!\vee\! P^4(T[c_2])\!\vee\!\big(\bigvee_{j=1}^{c_2}C^5_{r_j}\big)\!\vee\! C^6_\eta;
\end{multline*}
\item\label{nonspin:b} 
otherwise either there is a homotopy equivalence 
\begin{multline*}
  \Sigma M\simeq \big(\bigvee_{i=1}^lS^2\big)\!\vee\! \big(\bigvee_{i=1}^{d-c_1}S^3\big)\!\vee\! \big(\bigvee_{i=1}^{d}S^4\big)\!\vee\! \big(\bigvee_{i=1}^{l-c_1-c_2}S^5\big)\!\vee\! P^3(H)\!\vee\! P^5(H)\\\!\vee\! \big(\bigvee_{i=1}^{c_1} C^5_\eta\big)\!\vee\!\big(\bigvee_{j=1}^{c_2}C^5_{r_j}\big)\!\vee\!P^4\big(\frac{T[c_2]}{\z{r_{j_1}}}\big)\!\vee\! (P^4(2^{r_{j_1}})\cup_{\tilde{\eta}_{r_{j_1}}}e^6), 
\end{multline*}
or there is a homotopy equivalence 
\begin{multline*}
  \Sigma M\simeq \big(\bigvee_{i=1}^lS^2\big)\!\vee\! \big(\bigvee_{i=1}^{d-c_1}S^3\big)\!\vee\! \big(\bigvee_{i=1}^{d}S^4\big)\!\vee\! \big(\bigvee_{i=1}^{l-c_1-c_2}S^5\big)\!\vee\! P^3(H)\!\vee\! P^5(H)\\\!\vee\! \big(\bigvee_{i=1}^{c_1} C^5_\eta\big)\!\vee\! P^4(T[c_2])\!\vee\!\big(\bigvee_{j_1\neq j=1}^{c_2}C^5_{r_j}\big)\!\vee\! (C^5_{r_{j_1}}\cup_{i_P\tilde{\eta}_{r_{j_1}}}e^6), 
\end{multline*}
where $i_P\colon P^{5}(2^{r_{j_1}})\to C^{6}_{r_{j_1}}$ is the canonical inclusion map; in both cases, $r_{j_1}$ is the minimum of $r_j$ such that $u+v\in \im(\beta_{r_{j_1}})$.
\end{enumerate}
\end{enumerate}

\end{theorem}

In Theorem \ref{thm:main} we characterize the homotopy types of $\Sigma M$ by elementary complexes of dimension at most $6$, up to certain indeterminate $\an{2}$-complexes. Note that wedge summands of the form $\bigvee_{i=u}^vX$ with $v<u$ are contractible and can be removed from the homotopy decompositions of $\Sigma M$.  More generally, if $M$ is a $5$-dimensional \PD complex (i.e., a finite CW-complex whose integral cohomology satisfies the \PD theorem) satisfying the conditions in Theorem \ref{thm:main}, then Theorem \ref{thm:main} gives the homotopy types of $\Sigma M$, except that there are two additional possibilities when the Steenrod square acts trivially on $H^3(M;\z{})$, See Remark \ref{rmk:PDcpx}.

Due to Lemma \ref{lem:S5P3} (\ref{S5P3-susp}), the $3$-torsion of $H$ can be well understood when studying the homotopy types of the double suspension $\Sigma^2 M$. 
\begin{theorem}\label{thm:dbsusp}
  Let $M$ be an orientable smooth closed $5$-manifold with $H_\ast(M)$ given by (\ref{HM}), where $H$ is a $2$-torsion free group. Then the suspensions of the homotopy equivalences in Theorem \ref{thm:main} give the homotopy types of the double suspension $\Sigma^2M$. 
\end{theorem}

In addition to the characterization of the homotopy types of iterated loop spaces of the gauge groups of principal bundles over $M$, as shown by Huang \cite{Huang21}, we apply the homotopy types of $\Sigma M$ (or $\Sigma^2M$) to study the reduced $K$-groups and the cohomotopy sets $\pi^k(M)=[M,S^k]$ of the non-simply-connected manifold $M$. 

\begin{corollary}[See Proposition \ref{prop:kgrps}]\label{cor:kgrps}
  Let $M$ be a $5$-manifold given by Theorem \ref{thm:main} or \ref{thm:dbsusp}. Then the reduced complex $K$-group and $KO$-group of $M$ are given by 
  \[\wk(M)\cong\Z^{d+l}\oplus H\oplus H,\quad \ko(M)\cong\Z^l\oplus(\z{})^{l+d+t_2}.\]
\end{corollary}

The third cohomotopy set $\pi^3(M)$ possess the following property.
\begin{corollary}[See Proposition \ref{prop:MS3}]\label{Cor:MS3}
  Let $M$ be a $5$-manifold given by Theorem \ref{thm:main} or \ref{thm:dbsusp}. Then the suspension $\Sigma \colon \pi^3(M)\to \pi^4(\Sigma M)$ is a bijection. 
\end{corollary}

We also apply the homotopy decompositions of $\Sigma M$ to compute the group structure of $\pi^3(M)\cong \pi^4(\Sigma M)$, see Proposition \ref{prop:MS3}.
The second cohomotopy set $\pi^2(M)$ always admits an action of $\pi^3(M)$ induced by the Hopf map $\eta\colon S^3\to S^2$, see Lemma \ref{lem:2-cohtp} or \cite[Theorem 3]{KMT12}.
Finally, it should be noting that when $M$ is a $5$-dimensional \PD complex with $H_1(M)$ torsion free, similar results have been proved independently
and concurrently by Amelotte, Cutler and So \cite{ACS24}.

This paper is organized as follows. Section \ref{sec:prelim} reviews some homotopy theory of $\an{2}$-complexes and introduces the basic analysis methods to study the homotopy type of  homotopy cofibres. In Section \ref{sec:homoldecomp} we study the homotopy types of the suspension of the CW-complex $\overline{M}$ of $M$ with its top cell removed. The basic method is the homology decomposition of simply-connected spaces. Section \ref{sec:Proofs} analyzes the homotopy types of $\Sigma M$ and contains the proofs of Theorem \ref{thm:main} and \ref{thm:dbsusp}. As applications of the homotopy decomposition of $\Sigma M$ or $\Sigma^2M$, we study the reduced $K$-groups and the cohomotopy sets of the $5$-manifolds $M$ in Section \ref{sec:appl}. 

\subsection*{Acknowledgements}
The authors would like to thank the reviewer(s) for the new and faster proof of Lemma \ref{lem:S5P3} (\ref{S5P3-susp}).
Pengcheng Li and Zhongjian Zhu were supported by National Natural Science Foundation of China under Grant 12101290 and 11701430, respectively.

\section{Preliminaries}\label{sec:prelim}

Throughout the paper we shall use the following global conventions and notations. All spaces are based CW-complexes, all maps are base-point-preserving and are identified with their homotopy classes in notation. A strict equality is often treated as a homotopy equality. Denote by $\id_X$ the identity map of a space $X$ and simplify $\id_n=\id_{S^n}$. For different $X$, we use the ambiguous notations $i_k\colon S^k\to X$ and $q_k\colon X\to S^k$ to denote the possible canonical inclusion and pinch maps, respectively. For instance, there are inclusions 
$i_n\colon S^n\to C$ for each elementary Chang complex $C$ and there are inclusions $i_{n+1}\colon S^{n+1}\to X$ for $X=C^{n+2,s}$ and $C^{n+2,s}_r$. Let $i_P\colon P^{n+1}(2^r)\to C^{n+2}_r$ and $i_\eta\colon C^{n+2}_\eta\to C^{n+2}_r$ be the canonical inclusions. Denote by $C_f$ the homotopy cofibre of a map $f\colon X\to Y$.
For an abelian group $G$ generated by $x_1,\cdots,x_n$, denote $G\cong C_1\langle x_1\rangle\oplus\cdots\oplus C_n\langle x_n\rangle $ if $x_i$ is a generator of the cyclic direct summand $C_i$, $i=1,\cdots,n$. % For a prime $p$, denote by $\pi_i(X;p)$ the $p$-primary component of the $i$-th homotopy group $\pi_i(X)$.

\subsection{Some homotopy theory of $\mathbf{A}_n^2$-complexes}

For each prime $p$ and integers $r,s\geq 1,n\geq 2$, there exists a map (with $n$ omitted in notation) 
\[B(\chi^r_s)\colon P^{n+1}(p^r)\to P^{n+1}(p^s)\] 
satisfies $\Sigma B(\chi^r_s)=B(\chi^r_s)$  and the relation formulas (cf. \cite{BH91}):
  \begin{equation}\label{eq:chi}
   B(\chi^r_s)i_n=\chi^r_s\cdot i_n,\quad  q_{n+1}B(\chi^r_s)=\chi^s_r \cdot q_{n+1},
  \end{equation}
where $\chi^r_s$ is a self-map of spheres, $\chi^r_s=1$ for $r\geq s$ and $\chi^r_s=p^{s-r}$ for $r<s$.

\begin{lemma}\label{lem:Moore}
  Let $p$ be an odd prime and let $n\geq 3$, $r,s\geq 1$ be integers, $m=\min\{r,s\}$. There hold isomorphisms: 
  \begin{enumerate}
    \item\label{Moore-S3P3} $\pi_3(P^3(p^r))\cong\zp{r}\langle i_2\eta\rangle$ and $\pi_{n+1}(P^{n+i}(p^r))=0$, $i=0,1$. 
    \item\label{Moore-P3P3} $[P^n(p^r),P^n(p^s)]\cong \left\{\begin{array}{ll}
      \zp{m}\langle B(\chi^r_s)\rangle\oplus\zp{m}\langle i_2\eta q_3\rangle,& n=3;\\
      \zp{m}\langle B(\chi^r_s)\rangle,&n\geq 4.
    \end{array}\right.$ 
    \item\label{Moore-P4P3} $[P^{n+1}(p^r),P^n(p^s)]\cong  \left\{\begin{array}{ll}
      \zp{m}\langle\hat{\eta}_s B(\chi^r_s)\rangle,&n=3;\\
      0&n\geq 4.
    \end{array}\right.$  \\ where $\hat{\eta}_s\colon P^4(p^s)\to P^3(p^s)$ satisfies $\hat{\eta}_si_3=i_2\eta$. 
    
  \end{enumerate}
  \begin{proof}
The group $\pi_3(P^3(p^r))$ refers to \cite[Lemma 2.1]{ST19} and the groups $\pi_{n+1}(P^{n+i})=0$ was proved in \cite[Lemma 6.3 and 6.4]{Huang-6mflds}.  The groups and generators in
(2) and (3) can be easily computed by applying the exact functor $[-,P^n(p^s)]$ to the canonical cofibrations for $P^{n+i}(p^r)$ with $i=0,1$, respectively; the details are omitted here. 
  \end{proof}
\end{lemma}

\begin{lemma}[cf. \cite{BH91}]\label{lem:Moore1}
  Let $n\geq 3,r\geq 1$ be integers. 
 \begin{enumerate}
  \item $\pi_{n+1}(P^{n+1}(2^r))\cong \z{}\langle i_n\eta\rangle$.

  \item $\pi_{n+2}(P^{n+1}(2^r))\cong \left\{\begin{array}{ll}
    \Z/4\langle \tilde{\eta}_1\rangle,&r=1;\\
    \Z/2\oplus\z{}\langle \tilde{\eta}_r,i_n\eta^2\rangle,&r\geq 2.
  \end{array}\right.$
  \\
The generator $\tilde{\eta}_r$ satisfies formulas
\begin{equation}\label{eq:chi-eta}
  q_{n+1}\tilde{\eta}_r=\eta,\quad 2\tilde{\eta}_1=i_n\eta^2,\quad  B(\chi^r_s)\tilde{\eta}_r=\chi^s_r\cdot\tilde{\eta}_s.
\end{equation}
\item $[P^{n+1}(2^r),P^{n+1}(2^s)]\cong \left\{\begin{array}{ll}
  \Z/4\langle \id_P\rangle,&r=s=1;\\
  \z{m}\langle B(\chi^r_s)\rangle\oplus\z{}\langle i\eta q\rangle,&\text{otherwise},
\end{array}\right.$ \\
where $m=\min\{r,s\}$, $i\eta q=i_n\eta q_{n+1}$. 

 \end{enumerate}

\end{lemma}

\begin{lemma}\label{lem:S5P3}
  The following hold:
  \begin{enumerate}
    \item\label{S5P3-grps} $\pi_5(P^3(3^r))\cong \Z/3^{r+1}$, $\pi_5(P^3(p^r))=0$ for primes $p\geq 5$.
    \item\label{S5P3-susp} The suspension $\Sigma \colon \pi_5(P^3(3^r))\to \pi_6(P^4(3^r))$ is trivial.
  \end{enumerate}
  
\begin{proof}
  (1)  Let $F^3\{p^r\}$ be the homotopy fibre of $q_3\colon P^3(p^r)\to S^3$ and consider the induced exact sequence of $p$-local groups:
    \[\pi_6(S^3;p)\to \pi_5(F^3\{p^r\})\xra{(j_r)_{\sharp}} \pi_5(P^3(p^r))\xra{(q_3)_\sharp} \pi_5(S^3;p)=0.\]
 By \cite[Proposition 14.2]{Neisen80} or \cite[Theorem 3.1]{Neisen81}, there is a homotopy equivalence 
    \[\Omega F^3\{p^r\}\simeq S^1\times \prod_{j=1}^\infty S^{2p^j-1}\{p^{r+1}\}\times \Omega \big(\bigvee_\alpha P^{n_\alpha}(p^r)\big),\]
    where $S^{2n+1}\{p^r\}$ is the homotopy fibre of the mod $p^r$ degree map on $S^{2n+1}$, $n_\alpha\geq 4$ and the equality holds for exactly one $\alpha$. It follows that 
    \[ \pi_5(F^3\{p^r\})\cong \pi_4(S^{2p-1}\{p^{r+1}\})\cong \left\{\begin{array}{ll}
      \Z/3^{r+1},&p=3;\\
      0,&p\geq 5.
    \end{array}\right.\]
    Thus $\pi_5(P^3(p^r))=0$ for $p\geq 5$.  By \cite[Theorem 2.10]{Neisen81}, $\pi_5(P^3(3^r))$ contains a direct summand $\Z/3^{r+1}$, therefore we have an isomorphism
   \[(j_r)_\sharp\colon\pi_5(F^3\{3^r\})\xra{\cong}\pi_5(P^3(3^r))\cong\Z/3^{r+1}.\]

 (2) Firstly, by \cite{CMN79-b} for any prime $p\geq 5$ and \cite{Neisen81} for $p=3$, there is a homotopy equivalence 
 \[\Omega P^4(p^r)\simeq S^3\{p^r\}\times \Omega \big(\bigvee_{k=0}^{\infty}P^{7+2k}(p^r)\big).\]
Second, for skeletal reasons, the suspension $E\colon P^3(p^r)\to \Omega P^4(p^r)$ factors as the composite $P^3(p^r)\xra{i}S^3\{p^r\}\xra{j}\Omega P^4(p^r)$, where $i$ is the inclusion of the bottom Moore space and $j$ is the inclusion of a factor. Third, there is a homotopy fibration diagram 
\[\begin{tikzcd}
  E^3\{p^r\}\ar[d,equals]\ar[r]&F^3\{p^r\}\ar[d]\ar[r]&\Omega S^3\ar[d]\\
  E^3\{p^r\}\ar[r]&P^3(p^r)\ar[r,"i"]\ar[d,"q_3"]&S^3\{p^r\}\ar[d]\\
  &S^3\ar[r,equals]&S^3
\end{tikzcd}\]
that defines the space $E^3\{p^r\}$. By \cite{CMN79-a}, for any prime $p\geq 5$ and \cite{Neisen81} for $p=3$, there is a homotopy equivalence 
\[\Omega E^3\{p^r\}\simeq W_n\times \prod_{j=1}^{\infty}S^{2p^j-1}\{p^{r+1}\}\times \Omega \big(\bigvee_\alpha P^{n_{\alpha}}(p^r)\big),\]
where $W_n$ is the homotopy fibre of the double suspension. This decomposition has the property that the factor $\prod_{j=1}^{\infty}S^{2p^j-1}\{p^{r+1}\}$ of $\Omega F^3\{p^r\}$ may be chosen to factor through $\Omega E^3\{p^r\}$. 

Consequently, when $p=3$, as the $\Z/3^{r+1}$ factor in $\pi_4(\Omega P^3(p^r))$ came from $\pi_4(\prod_{j=1}^{\infty}S^{2p^j-1}\{p^{r+1}\})$, it has the property that it composes trivially with the map $\Omega i\colon \Omega P^3(3^r)\to \Omega S^3\{3^r\}$. Hence, as $\Omega E$ factors through $\Omega i$, the $\Z/3^{r+1}$ factor in $\pi_4(\Omega P^3(p^r))$ composes trivially with $\Omega E$. Thus the  $\Z/3^{r+1}$ factor in $\pi_5(P^3(p^r))$ suspends trivially.
  \end{proof}
\end{lemma}

\begin{lemma}[cf. \cite{lipc2-an2}]\label{lem:Changcpx}
  Let $n\geq 3$ and $r\geq 1$. There hold isomorphisms
\begin{enumerate}
  \item\label{Changcpx-1} $\pi_{n+2}(C^{n+2}_\eta)\cong \Z\langle  \tilde{\zeta}\rangle$, where $\tilde{\zeta}$ satisfies $q_{n+2}\tilde{\zeta}=2\cdot \id_{n+2}$.
  \item\label{Changcpx-2} $\pi_{n+2}(C^{n+2}_r)\cong \Z\langle i_\eta \tilde{\zeta}\rangle\oplus \z{}\langle i_P\tilde{\eta}_r\rangle$.
\end{enumerate}
It follows that  a map $f_C\colon S^{n+2}\to C$ with $C=C^{n+2}_\eta$ or $C^{n+2}_r$ induces the trivial homomorphism in integral homology if and only if 
\[f_C=\left\{\begin{array}{ll}
  0&\text{ for }C=C^{n+2}_\eta;\\
  0 \text{ or }i_P\tilde{\eta}_r & \text{ for }C=C^{n+2}_r,
\end{array}\right.\] 
where $f=0$ means $f$ is null-homotopic.
\end{lemma}

The following Lemma can be found in \cite[Theorem 3.1, (2)]{lipc2-an2}; since it hasn't been published yet, we give a proof here.  

\begin{lemma}\label{lem:eq:C_r}
For integers $n\geq 3$ and $r\geq 1$, there exists a map 
\[\bar{\xi}_r\colon C^{n+2}_r\to P^{n+1}(2^{r+1})\] 
satisfying the homotopy commutative diagram of homotopy cofibrations
  \[\begin{tikzcd}
 S^{n}\ar[r,"i_n2^r"]\ar[d,equal] & C^{n+2}_\eta \ar[r,"i_\eta"]\ar[d,"\bar{\zeta}"]& C^{n+2}_r\ar[d,"\bar{\xi}_r"]\ar[r,"q_{n+1}"]&S^{n+1}\ar[d,equal]\\
 S^{n}\ar[r,"2^{r+1}"]& S^n\ar[r,"i_n"]& P^{n+1}(2^{r+1})\ar[r,"q_{n+1}"]&S^{n+1}
\end{tikzcd}.\]
Moreover, there hold formulas  
\begin{equation}\label{eq:C_r-P}
  \bar{\xi}_r\circ i_P=B(\chi^r_{r+1}),\quad B(\chi^{s+1}_r)\bar{\xi}_s (i_P\tilde{\eta}_s)= \tilde{\eta}_r \text{ for }r>s.
\end{equation}
 
\begin{proof}
Dual to the relation in Lemma \ref{lem:Changcpx} (\ref{Changcpx-1}), there exists a map $\bar{\zeta}\colon C^{n+2}_\eta\to S^n$ satisfying $\bar{\zeta}i_n=2\cdot \id_n.$
It follows that the first square in the Lemma is homotopy commutative, and hence the map $\bar{\xi}_r$ in the Lemma exists. Recall we have the composition 
\[i_n=i_\eta \circ i_n \colon S^n\to C^{n+2}_\eta\to C^{n+2}_r.\]
Then $\bar{\xi}_ri_n=(\bar{\xi}_ri_\eta) i_n=(i_n\bar{\zeta})i_n=2i_n$
implying that 
\[\bar{\xi}_r\circ i_P=B(\chi^r_{r+1})+\varepsilon\cdot i_n\eta q_{n+1}\] for some $\varepsilon \in\{0,1\}.$
If $\varepsilon=0$, we are done; otherwise we replace $\bar{\xi}_r$ by $\bar{\xi}_r+i_n\eta q_{n+1}$ to make $\varepsilon=0$. Note that all the relations mentioned above still hold even if we make such a replacement. Thus we prove the first formula in (\ref{eq:C_r-P}), which implies the second one. 
\end{proof}
\end{lemma}

\subsection{Basic analysis methods}
We give some auxiliary lemmas that are useful to study the homotopy types of homotopy cofibres. 

\begin{lemma}\label{lem:cupprod}
  Let $C_k^X$ be the homotopy cofibre of $f^X_k\colon  X\to P^3(p^s)$, where $k\in\zp{\min\{r,s\}}$ and  $r=\infty$ for $X=S^3$,
  \[f^X_k=\left\{\begin{array}{ll}
    k\cdot i_2\eta,& X=S^3;\\
    k\cdot i_2\eta q_3,& X=P^3(p^r).
  \end{array}\right.\]
 Then the cup squares in $H^\ast(C_k^X;\zp{\min\{r,s\}})$ are given by 
  \[u_2\smallsmile u_2=k\cdot u_4,\] 
  where $u_i\in H^i(C_k^X;\zp{\min\{r,s\}})$ are generators, $i=2,4$. It follows that all cup squares in $H^\ast(C_k^X;\zp{\min\{r,s\}})$ are trivial if and only if $k=0$.
  \begin{proof}
    It is well-known that the map $k\eta$ has Hopf invariant $H(k\eta)=kH(\eta)=k$. Let $m=\min\{r,s\}$ and define $u_2\smallsmile u_2=\bar{H}(f_k^X)\cdot u_4$ for some $\bar{H}(f_k^X)\in\zp{m}$, which is called the \emph{mod $p^m$ Hopf invariant}. Then by naturality it is easy to deduce the formula 
\[\bar{H}(f_k^X)= H(k\eta)\pmod {p^m}=k,\]
which completes the proof of the Lemma.
  \end{proof}
\end{lemma}

\begin{lemma}\label{lem:cupprod2}
 Let $k\in \zp{\min\{r,s\}}$ and consider the homotopy cofibration \[P^4(p^r)\xra{g_k=k\cdot \hat{\eta}_sB(\chi^r_s)} P^3(p^s)\to C_{g_k}.\]  
 Let $v_i$ be generators of $ H^i(C_{g_k};\zp{s})$, $i=2,4$, then   
 \[v_2\smallsmile v_2=k\cdot v_4\in H^4(C_{g_k};\zp{s})\cong\zp{\min\{r,s\}}.\]
 It follows that $g_k$ is null-homotopic if and only if $k=0$.
 \begin{proof}
By Lemma \ref{lem:Moore} (\ref{Moore-P4P3}), there is a homotopy commutative diagram of homotopy cofibrations
\[ \begin{tikzcd}
  S^3\ar[rr,"{k\chi^r_s\cdot i_2\eta}"]\ar[d,"i_3"]&&P^3(p^s)\ar[r]\ar[d,equal]&C_{k\chi^r_s}\ar[d,"\imath"]\\
  P^4(p^r)\ar[d,"q_4"]\ar[rr,"k\cdot \hat{\eta}_sB(\chi^r_s)"]&&P^3(p^s)\ar[d]\ar[r]&C_{g_k}\ar[d] \\
  S^4\ar[rr]&&\ast\ar[r]&S^5
  \end{tikzcd}.\]
It follows that $\imath$ in the right-most column induces an isomorphism
\[H^2(C_{g_k};\zp{s})\xra[\cong]{\imath^\ast} H^2(C_{k\chi^r_s};\zp{s})\cong\zp{s}\]
and a monomorphism 
\[ H^4(C_{g_k};\zp{s})\cong \zp{\min\{r,s\}}\xra{ \imath^\ast} H^4(C_{k\chi^r_s};\zp{s})\cong\zp{s}.\]
Let $v_i\in H^i(C_{g_k};\zp{s})$ be generators, $i=2,4$; let $u_2=\imath^\ast(v_2)$ and $u_4$ be generators of $H^2(C_{k\chi^r_s};\zp{s})$ and $H^4(C_{k\chi^r_s};\zp{s})$, respectively. 
Let $\bar{H}(g_k)$ be the mod $p^s$ Hopf invariant of $g_k$. By the naturality of cup products and Lemma \ref{lem:cupprod}, we have 
\begin{align*}
 k\chi^r_s \cdot u_4&=u_2\smallsmile u_2=\imath^\ast(v_2\smallsmile v_2)=\imath^\ast(\bar{H}(g_k)v_4)=\bar{H}(g_k) \cdot (\chi^r_s\cdot u_4).
\end{align*}
Thus $\bar{H}(g_k)=k$, which completes the proof.
 \end{proof}
\end{lemma}

The method of proof for the following lemma is due to \cite[Lemma 2.4]{CS22}.
\begin{lemma}\label{lem:cupprods=0=nullhtp}
  Let $X_1,X_2\in\{S^2,P^3(2^r),C^4_s\}$ with $r,s\geq 1$. Let \[\iota_1\colon \Sigma X_1\to \Sigma X_1\vee \Sigma X_2, \quad \iota_2\colon \Sigma X_1\to \Sigma X_2\vee \Sigma X_2\] be the canonical inclusion maps. 
  Then any map $u'$ in the composition 
  \[u\colon S^5\xra{u'}\Sigma X_1\wedge X_2\xra{[\iota_1,\iota_2]}\Sigma X_1\vee \Sigma X_2\]
  is null-homotopic if and only if all cup products in $H^\ast(C_u;G)$ are trivial, where $C_u$ is the homotopy cofibre of $u$ and $G=H_2(X_1)\otimes H_2(X_2)$.
  \begin{proof}
    The ``only if'' part is clear. For the ``if'' part, consider the following homotopy commutative diagram of homotopy cofibrations
    \begin{equation*}
      \begin{tikzcd}
        S^5\ar[r,"u'"]\ar[d,equal]&\Sigma X_1\wedge X_2\ar[d,"{[\iota_1,\iota_2]}"]\ar[r,"i'"]&C_{u'}\ar[d]\\
        S^5\ar[r,"u"]\ar[d]&\Sigma X_1\vee \Sigma X_2\ar[d,"j"]\ar[r,"i"]&C_{u}\ar[d,"\bar{j}"]\\
        \ast\ar[r]&\Sigma X_1\times\Sigma X_2\ar[r,equal]&\Sigma X_1\times \Sigma X_2
      \end{tikzcd},
    \end{equation*}
which induces the commutative diagram with exact rows and columns:
\[\begin{tikzcd}
  H^5(C_{u'};G)\ar[r,"(i')^\ast"]\ar[d,"\delta_1"]&H^5(\Sigma X_1\wedge X_2;G)\ar[d,"\delta_2"]\ar[r,"(u')^\ast"]&H^5(S^5;G)\\
H^6(\Sigma X_1\times\Sigma X_2;G)\ar[r,equal]\ar[d,"\bar{j}^\ast"]&H^6(\Sigma X_1\times\Sigma X_2;G)\ar[d]\\
H^6(C_u;G)\ar[r]&H^6(\Sigma X_1\vee \Sigma X_2;G)=0
\end{tikzcd}.\]
Note that $H^6(\Sigma X_1\times\Sigma X_2;G)$ is generated by cup products, while all cup products in $H^6(C_u;G)$ are trivial by assumption. It follows that $\bar{j}^\ast=0$ and hence $\delta_1$ is surjective. The homomorphism $\delta_2$ is obviously an isomorphism for $X_1,X_2\in\{S^2,P^3(2^r)\}$ because $H^5(\Sigma X_1\vee \Sigma X_2;G)=0$; for $X_2=C^4_s$, $X_1=S^2,P^3(2^r)$ or $C^4_r$, we have $H^j(C^4_s;G)\cong G$ for $j=2,3,4$, where $G=\z{s}$ or $\z{\min\{r,s\}}$. By computations,
 \begin{align*}
  H^5(\Sigma X_1\wedge C^4_s;G)&\cong \bigoplus_{i+j=5}\tilde{H}^i(\Sigma X_1;\tilde{H}^{j}(C^4_s;G))\cong H^3(\Sigma X_1;H^2(C^4_s;G)),\\ H^6(\Sigma X_1\times C^5_s;G)&\cong \bigoplus_{i+j=6} H^i(\Sigma X_1;H^{j}(C^5_s;G))\cong H^3(\Sigma X_1;H^3(C^5_s;G)). 
 \end{align*}
Thus $\delta_2$ is an isomorphism for all $X_1,X_2$. The upper commutative square then implies that $(i')^\ast$ is surjective and therefore $(u')^\ast$ is the zero map by exactness.
Since $\Sigma X_1\wedge X_2$ is $4$-connected, the universal coefficient theorem for cohomology implies that
\[0=(u')_\ast\colon H_5(S^5)\to H_5(\Sigma X_1\wedge X_2).\] 
Therefore $u'$ is null-homotopic, by the Hurewicz theorem.
  \end{proof}
\end{lemma}

\begin{lemma}\label{lem:Sq2-Chang}
 The Steenrod square $\Sq^2\colon H^{n}(C;\z{})\to H^{n+2}(C;\z{})$
is an isomorphism for every $(n+2)$-dimensional elementary Chang complex $C$.
\begin{proof}
  Obvious or see \cite{ZP17}.
\end{proof}
\end{lemma}

For $n\geq 3$ and $r\geq 1$, we define homotopy cofibres
\begin{equation}\label{eq:An+3}
   A^{n+3}(\tilde{\eta}_r)=P^{n+1}(2^r)\cup_{\tilde{\eta}_r}e^{n+3},~~A^{n+3}(i_P\tilde{\eta}_r)=C^{n+2}_r\cup_{i_P\tilde{\eta}_r}e^{n+3}.
\end{equation}

\begin{lemma}\label{lem:Sq2}
  The Steenrod square $\Sq^2\colon H^{n+1}(X;\z{})\to H^{n+3}(X;\z{})$
  is an isomorphism for $X= A^{n+3}(\tilde{\eta}_r)$ and $A^{n+3}(i_P\tilde{\eta}_r)$.
  \begin{proof}
    The statement for $X= A^{n+3}(\tilde{\eta}_r)$ refers to \cite[Lemma 2.6]{lipc23}. For $X=A^{n+3}(i_P\tilde{\eta}_r)$, consider the homotopy commutative diagram of homotopy cofibrations
    \[\begin{tikzcd}
      S^{n+2}\ar[d,equal]\ar[r,"\tilde{\eta}_r"]&P^{n+1}(2^r)\ar[d,"i_P"]\ar[r]&A^{n+3}(\tilde{\eta}_r)\ar[d,"\imath"]\\
      S^{n+2}\ar[d]\ar[r,"i_P\tilde{\eta}_r"]&C^{n+2}_r\ar[r]\ar[d,"q_{n+2}"]&A^{n+3}(i_P\tilde{\eta}_r)\ar[d]\\
      \ast\ar[r]&S^{n+2}\ar[r,equal]&S^{n+2}
    \end{tikzcd}.\]
    From the first two rows of the homotopy commutative diagram, it is easy to compute that 
  \[ H^{n+i}(A^{n+3}(\tilde{\eta}_r);\Z/2)\cong H^{n+i}(A^{n+3}(i_P\tilde{\eta}_r);\Z/2)\cong\Z/2 \text{ for $i=1,3$}.\]
  The third column homotopy cofibration implies that the induced homomorphisms $\imath^\ast$ are monomorphisms of mod $2$ homology groups of dimension $n+1$ and $n+3$, hence it is an isomorphism.
Then we complete the proof by the naturality of $\Sq^2$.
  \end{proof}
\end{lemma}

  \begin{lemma}[Lemma 6.4 of \cite{HL}]\label{lem:HL}
    Let $S\xra{f}(\bigvee_{i=1}^nA_i)\vee B \xra{g}\Sigma C$ be a homotopy cofibration of simply-connected CW-complexes. For each $j=1,\cdots, n$, let \[p_j\colon \big(\bigvee_{i}A_i\big)\vee B\to A_j,\quad q_B\colon \big(\bigvee_{i}A_i\big)\vee B\to B\] be the obvious projections. Suppose that the composite $p_jf$
  is null-homotopic for each $j\leq n$, then there is a homotopy equivalence
  \[\Sigma C\simeq \big(\bigvee_{i=1}^nA_i\big)\vee C_{q_Bf},\]
  where $C_{q_Bf}$ is the homotopy cofibre of the composite $q_Bf$.
  \end{lemma}

  \begin{lemma}\label{lem:splitting}
    Let $(\bigvee_{i=1}^nA_i)\vee B \xra{f}C\to D$ be a homotopy cofibration of CW-complexes. If the restriction of $f$ to $A_i$ is null-homotopic for each $i=1,\cdots,n$, then there is a homotopy equivalence 
    \[D\simeq \big(\bigvee_{i=1}^n\Sigma A_i\big)\vee E,\]
    where $E$ is the homotopy cofibre of the restriction $f|B\colon B\to C$.
    \begin{proof}
      Clear.
    \end{proof}
  \end{lemma}

Let $X=\Sigma X'$, $Y_i=\Sigma Y_i'$ be suspensions, $i=1,2,\cdots,n$. Let \[i_l\colon Y_l\to \bigvee_{j=i}^nY_i,\quad p_k\colon \bigvee_{i=1}^n Y_i\to Y_k\] be respectively the canonical inclusions and projections, $1\leq k,l\leq n$. By the Hilton-Milnor theorem,  we may write a map
  \(f\colon X\to\bigvee_{i=1}^nY_i\)
  as \[f=\sum_{k=1}^n i_k\circ f_{k}+\theta,\]
  where $f_{k}=p_k\circ f\colon X\to Y_k$ and $\theta$ satisfies $\Sigma \theta=0$. The first part $\sum_{k=1}^n i_k\circ f_{k}$ is usually represented by a vector $u_f=(f_1,f_2,\cdots,f_n)^t.$ 
  We say that $f$ is completely determined by its components $f_k$ if $\theta=0$; in this case, denote $f=u_f$. Let $h=\sum_{k,l}i_lh_{lk}p_k$ be a self-map of $\bigvee_{i=1}^nY_i$ which is completely determined by its components $h_{kl}=p_k\circ h\circ i_l\colon Y_l\to Y_k$. Denote by   
  \[M_h\coloneqq (h_{kl})_{n\times n}=\begin{bmatrix}
    h_{11}&h_{12}&\cdots&h_{1n}\\
    h_{21}&h_{22}&\cdots&h_{2n}\\
    \vdots&\vdots&\ddots&\vdots\\
    h_{n1}&h_{n1}&\cdots&h_{nn}
  \end{bmatrix}\] 
  Then the composition law
  \(h(f+g)\simeq h f+h g\)
  implies that the product
  \[M_h[f_1,f_2,\cdots,f_n]^t\]
  given by the matrix multiplication represents the composite $h\circ f$.
  Two maps $f=u_f$ and $g=u_g$ are called \emph{equivalent}, denoted by 
  \[[f_1,f_2,\cdots,f_n]^t\sim [g_1,g_2,\cdots,g_n]^t,\]
  if there is a self-homotopy equivalence $h$ of $\bigvee_{i=1}^n Y_i$, which can be represented by the matrix $M_h$, such that 
  \[M_h[f_1,f_2,\cdots,f_n]^t\simeq [g_1,g_2,\cdots,g_n]^t.\] 
  Note that the above matrix multiplication refers to elementary row operations in matrix theory; and the homotopy cofibres of the maps $f=u_f$ and $g=u_g$ are homotopy equivalent if $f$ and $g$ are equivalent.

\section{Homology Decomposition of $\Sigma M$}\label{sec:homoldecomp}

Recall the homology decomposition of a simply-connected space $X$ (cf. \cite[Theorem 4H.3]{hatcherbook}). For $n\geq 2$, the \emph{$n$-th homology section} $X_n$ of $X$ is a CW-complex constructed from $X_{n-1}$ by attaching a cone on a Moore space $M(H_nX,n-1)$;  by definition, $X_1=\ast$. Note that for each $n\geq 2$, there is a canonical map $j_n\colon X_n\to X$ that induces an isomorphism $j_{n\ast}\colon H_r(X_n)\to H_r(X)$ for $r\leq n$ and $H_r(X_n)=0$ for $r>n$.

%In this section we determine the homotopy types of the first five homology sections of $\Sigma M$.
Firstly we note that similar arguments to the proof of \cite[Lemma 5.1]{ST19} proves the following lemma.
\begin{lemma}\label{lem:ST}
  Let $M$ be an orientable closed manifold with $H_1(M)\cong\Z^l\oplus H$, where $l\geq 1$ and $H$ is a torsion abelian group. Then there is a homotopy equivalence 
\[\Sigma M\simeq \bigvee_{i=1}^l S^2\vee \Sigma W,\]
where $W=M/\bigvee_{i=1}^l S^1$ is the quotient space with $H_1(W)\cong H$.
\end{lemma}
By Lemma \ref{lem:ST} and (\ref{HM}), the homology groups of $\Sigma W$ is given by  
\begin{equation}\label{HSW}
  \begin{tabular}{ccccccc}
    \toprule 
    $i$&$2$&$3$&$4$&$5$&$0,6$& \text{otherwise} \\
    \midrule 
    $H_i(\Sigma W)$&$H$&$\Z^d\oplus T$&$\Z^d\oplus H$&$\Z^l$&$\Z$&$0$
   \\ \bottomrule
  \end{tabular}
\end{equation}
Let $W_i$ be the $i$-th homology section of $\Sigma W$. There are  homotopy cofibrations in which the attaching maps are \emph{homologically trivial} (induce trivial homomorphisms in integral homology):
\begin{equation}\label{Cofs}
  \begin{aligned}
   & \big(\bigvee_{i=1}^dS^2\big)\vee P^3(T)\xra{f} P^3(H)\to W_3,\\ 
   &\big(\bigvee_{i=1}^dS^3\big)\vee P^4(H)\xra{g} W_3\to W_4,\\
   &\bigvee_{i=1}^lS^4\xra{h} W_4\to W_5,\quad S^5\xra{\phi} W_5\to \Sigma W.
  \end{aligned}
\end{equation} 
%Note that there are natural inclusions $\Sigma W_{(i)}\subseteq W_{i+1}\subseteq \Sigma W_{(i+1)}$, where $W_{(i)}$ denotes the $i$-th skeleton of $W$. 
From now on we assume that $H\cong \bigoplus_{j=1}^h\Z/q_j^{s_j}$ where  $q_j$ are odd primes and $s_j\geq 1$.

\begin{lemma}\label{lem:W3}
There is a homotopy equivalence 
  \[W_3\simeq \big(\bigvee_{i=1}^dS^3\big)\vee P^3(H)\vee P^4(T).\]
\begin{proof}
 It suffices to show the map $f$ in (\ref{Cofs}) is null-homotopic, or equivalently the following components of $f$ are null-homotopic:
 \begin{align*}
  f^{S}&\colon \bigvee_{i=1}^dS^2\hookrightarrow \big( \bigvee_{i=1}^dS^2\big)\vee P^3(T)\xra{f}P^3(H),\\
  f^{T}&\colon P^3(T) \hookrightarrow \big(\bigvee_{i=1}^dS^2\big)\vee P^3(T)\xra{f}P^3(H),
 \end{align*}
where $\hookrightarrow$ denote the canonical inclusion maps. $f$ is homologically trivial, so are $f^{S}$ and $f^{T}$. Then the Hurewicz theorem and Lemma \ref{lem:Moore} (\ref{Moore-S3P3}) imply $f^{S}$ is null-homotopic.

Since $[P^3(p^r),P^3(q^s)]=0$ for different primes $p,q$, it suffices to consider the case where $T$ and $H$ have the same prime factors. Denote by $T_H\cong \bigoplus_{j}\Z/q_j^{r_j}$  the component of $T$ that has the same prime factors with $H$. 
The canonical inclusion $\imath_3\colon W_3\to \Sigma W$ induces an isomorphism with $m_j=\min\{r_j,s_j\}$:
\[\imath_3^\ast\colon H^2(\Sigma W;\Z/q_j^{m_j})\to H^2(W_3;\Z/q_j^{m_j}).\]
It follows that all the cup squares of cohomology classes of $H^2(W_3;\Z/q_j^{m_j})$, and hence of $H^2(C_{f^T};\Z/q_j^{m_j})$ are trivial for any $j$. Let $C_{f^T_j}$ be the homotopy cofibre of the compositions 
\[f^T_j\colon P^3(q_j^{r_j})\hookrightarrow P^3(T)\xra{f^T}P^3(H)
\twoheadrightarrow P^3(q_j^{s_j}),\]
where the unlabelled maps are the canonical inclusions and projections, respectively. 
Then \cite[Lemma 4.2]{ST19} implies that all cup squares of cohomology classes of $H^2(C_{f^T_j};\Z/q_j^{m_j})$ are trivial for any $j$ and hence $f^T_j$ is null-homotopic, by Lemma \ref{lem:cupprod}. Therefore $f^T$ is also null-homotopic and we complete the proof.
\end{proof}
\end{lemma}

\begin{lemma}\label{lem:W4}
 There is a homotopy equivalence 
\[W_4\simeq \big(\bigvee_{i=1}^d(S^3\vee S^4)\big)\vee P^3(H)\vee P^5(H)\vee P^4(T).\]

\begin{proof}
By (\ref{Cofs}) and Lemma \ref{lem:W3}, $W_4$ is the homotopy cofibre of a homologically trivial map 
\[\bar{g}\colon \big(\bigvee_{i=1}^dS^3\big)\vee P^4(H)\xra{g} W_3\xra[\simeq]{e} \big(\bigvee_{i=1}^dS^3\big)\vee P^3(H)\vee P^4(T). \]  
Consider the compositions 
\begin{align*}
 &S^3\hookrightarrow  \big(\bigvee_{i=1}^dS^3\big)\vee P^4(H)\xra{g}W_3\to \bigvee_{i=1}^dS^3\to S^3,\\
 &S^3\hookrightarrow  \big(\bigvee_{i=1}^dS^3\big)\vee P^4(H)\xra{g}W_3\to P^4(T),\\
&P^4(q_j^{s_j}) \hookrightarrow  \big(\bigvee_{i=1}^dS^3\big)\vee P^4(H)\xra{g}W_3\to \bigvee_{i=1}^dS^3\to S^3,\\
&P^4(q_j^{s_j}) \hookrightarrow  \big(\bigvee_{i=1}^dS^3\big)\vee P^4(H)\xra{g}W_3\to P^4(T)\to P^4(q_j^{r_j}),
\end{align*}
where the unlabelled maps are the canonical inclusions and projections. Since $[P^4(p^r),S^3]=0$, the Hurewicz theorem and Lemma \ref{lem:Moore} (\ref{Moore-P3P3}) imply that all the above compositions are null-homotopic. Hence by Lemma \ref{lem:HL} there is a homotopy equivalence
\[W_4\simeq \big(\bigvee_{i=1}^dS^3\big)\vee P^4(T)\vee C_{g'}\] 
for some map $g'\colon \big(\bigvee_{i=1}^dS^3\big)\vee P^4(H)\to P^3(H)$.

By the homology decomposition for $\Sigma W$ and the universal coefficient theorem for cohomology, the canonical map $\imath_4\colon W_4\to \Sigma W$ induces isomorphisms 
\[\imath_4^\ast\colon H^i(\Sigma W)\to H^i(W_4),~i=2,4.\]
Consider the commutative diagram 
\[\begin{tikzcd}
  H^2(\Sigma W;\Z/q_j^{s_j})\ar[r,"\smallsmile^2"]\ar[d,"\imath_4^\ast","\cong"swap]&H^4(\Sigma W;\Z/q_j^{s_j})\ar[d,"\imath_4^\ast","\cong"swap]\\
  H^2(W_4;\Z/q_j^{s_j})\ar[r,"\smallsmile^2"]&H^4(W_4;\Z/q_j^{s_j})
\end{tikzcd},\]
where $\smallsmile^2$ denotes the cup squares. All cup squares in $H^\ast(\Sigma W;\Z/q_j^{s_j})$ are trivial implying that all cup squares in $H^4(W_4;\Z/q_j^{s_j})$ are trivial. Let $C_{g'_j}$ and $C_{g'_{ij}}$ be the homotopy cofibres of the compositions 
\begin{align*}
  g'_j&\colon S^3\hookrightarrow\big(\bigvee_{i=1}^dS^3\big)\vee P^4(H)\xra{g'} P^3(H)\twoheadrightarrow P^3(q_j^{s_j}),\\
  g'_{ij}&\colon P^4(q_j^{r_i})\hookrightarrow\big(\bigvee_{i=1}^dS^3\big)\vee P^4(H)\xra{g'} P^3(H)\twoheadrightarrow P^3(q_j^{s_j}).
\end{align*}
By \cite[Lemma 4.2]{ST19}, we get the triviality of cup squares in $H^\ast(C_{g'_j};\Z/q_j^{s_j})$ and $H^\ast(C_{g'_{ij}};\Z/q_j^{s_j}))$. Then  Lemma \ref{lem:cupprod} and \ref{lem:cupprod2} imply that $g_j'$ and $g'_{ij}$ are both null-homotopic. Thus by Lemma \ref{lem:splitting}, there is a homotopy equivalence 
\[C_{g'}\simeq \big(\bigvee_{i=1}^dS^4\big)\vee P^3(H)\vee P^5(H),\]
which completes the proof of the Lemma.
\end{proof}
\end{lemma}

\begin{proposition}\label{prop:W5}
  There is a homotopy equivalence 
 \begin{multline*}
  W_5\simeq P^3(H)\!\vee\! P^5(H)\!\vee\!  P^4(T_{\neq 2})
  \!\vee\! \big(\bigvee_{i=1}^{d-c_1}S^3\big)\!\vee\!  \big(\bigvee_{i=1}^{d}S^4\big) \!\vee\!  \big(\bigvee_{i=1}^{l-c_1-c_2}S^5\big)\\\!\vee\! \big(\bigvee_{i=1}^{c_1}C^5_\eta\big)\!\vee\! \big(\bigvee_{j=c_2+1}^{t_2}P^4(2^{r_j})\big)\!\vee\! \big(\bigvee_{j=1}^{c_2}C^5_{r_j}\big),
 \end{multline*}
 where $0\leq c_1 \leq \min\{l,d\}$ and $0\leq c_2\leq \min\{l-c_1,t_2\}$; $c_1=c_2=0$ if and only if $\Sq^2\big(H^2(M;\z{})\big)=0$. 

 %In particular, if $l=1$, then $(c_1,c_2)\in \{(0,0),(1,0),(0,1)\}$.
\begin{proof}
  By (\ref{Cofs}) and Lemma \ref{lem:W4}, $W_5$ is the homotopy cofibre of a map 
  \[\bigvee_{i=1}^lS^4\xra{h}W_4\simeq \big(\bigvee_{i=1}^d(S^3\vee S^4)\big)\vee P^3(H)\vee P^5(H)\vee P^4(T).\]
Similar arguments to that in the proof of Lemma \ref{lem:W4} show that there is a homotopy equivalence 
\begin{equation}\label{eq:W5}
  W_5\simeq \big(\bigvee_{i=1}^d S^4\big)\vee P^3(H)\vee P^5(H)\vee P^4(T_{\neq 2})\vee C_{h'},
\end{equation}
where $h'\colon \bigvee_{i=1}^lS^4\to \big(\bigvee_{i=1}^dS^3\big)\vee \big(\bigvee_{i=1}^{t_2}P^4(2^{r_i})\big)$. 

Since $\pi_4(P^4(2^r))\cong \z{}\langle i_3\eta\rangle$, we may represent the map $h'$ by a $(d+t_2)\times l$-matrix $M_{h'}$ with entries $0$, $\eta$ or $i_3\eta$. There hold homotopy equivalences 
\begin{align*}
  \begin{bmatrix}
    \id_3&0\\
   i_3&\id_P
  \end{bmatrix}\matwo{\eta}{i_3\eta}&\simeq \matwo{\eta}{0}\colon S^4\to S^3\vee P^4(2^r),\\
  \begin{bmatrix}
    \id_P&0 \\
   B(\chi^r_s) &\id_P
  \end{bmatrix}\matwo{i_3\eta}{i_3\eta}&\simeq \matwo{i_3\eta}{0}\colon S^4\to P^4(2^r)\vee P^4(2^s) \text{ for $r\geq s$}.
\end{align*}
Then by elementary matrix operations we have an equivalence 
\[M_{h'}\sim \begin{bmatrix}
  D_{c_1}&O\\
  O&O\\
  O&\mat{E_{c_2}}{O}{O}{O}
\end{bmatrix},\]
where $O$ denote suitable zero matrices, $D_{c_1}$ is the diagonal matrix of rank $c_1$ whose diagonal entries are $\eta$,  $E_{c_2}$ is a $c_2\times c_2$-matrix which has exactly one entry $i_3\eta$ in each row and column. It follows that there is a homotopy equivalence 
\[C_{h'}\simeq \big(\bigvee_{i=1}^{l-c_1-c_2}S^5\big)\vee  \big(\bigvee_{i=1}^{d-c_1}S^3\big)\vee \big(\bigvee_{j=c_2+1}^{t_2}P^4(2^{r_j})\big)\vee \big(\bigvee_{i=1}^{c_1}C^5_\eta\big)\vee \big(\bigvee_{j=1}^{c_2}C^5_{r_j}\big).\]
The proof of the Lemma then follows by (\ref{eq:W5}) and Lemma \ref{lem:Sq2-Chang}.
\end{proof}
\end{proposition}

\section{Proof of Theorem \ref{thm:main} and \ref{thm:dbsusp}}\label{sec:Proofs}
Let $M$ be the given $5$-manifold described in Theorem \ref{thm:main}.
By (\ref{Cofs}) there is a homotopy cofibration 
$S^5\xra{\phi}W_5\to \Sigma W$ with $W_5$ (and integers $c_1,c_2$) given by Proposition \ref{prop:W5}. Since $\phi$ is homologically trivial, so are the compositions 
\begin{align*}
  \phi_\eta&\colon S^5\xra{\phi}W_5\twoheadrightarrow \bigvee_{i=1}^{c_1} C^5_\eta\twoheadrightarrow C^5_\eta,\\
  \phi_{C_j}&\colon   S^5\xra{\phi}W_5\twoheadrightarrow \bigvee_{j=1}^{c_2} C^5_{r_j}\twoheadrightarrow C^5_{r_j},\\
  \phi_{H,j}&\colon S^5\xra{\phi}W_5\twoheadrightarrow P^3(H)\twoheadrightarrow P^3(q_j^{s_j}).
\end{align*} 
By Lemma \ref{lem:Changcpx},  $\phi_\eta$ is null-homotopic and $\phi_{C_j}=w_j\cdot i_P\tilde{\eta}_{r_j}$ for some $w_j\in\z{}$. 
By Lemma \ref{lem:S5P3}, $\phi_{H,j}$ is null-homotopic for primes $q_j\geq 5$ and $\Sigma \phi_{H,j}$ are null-homotopic for all odd primes $q_j$. Write $H=H_3\oplus H_{\geq 5}$ with $H_3$ the $3$-primary component of $H$.
It follows by Lemma \ref{lem:Moore} (\ref{Moore-P3P3}) and \ref{lem:HL} that 
there are homotopy equivalences  
\begin{align}
 \Sigma W&\!\simeq\! P^3(H_{\geq 5})\!\vee\! P^5(H)\!\vee\! P^4(T_{\neq 2})\!\vee\!\big(\bigvee_{i=1}^{l-c_1-c_2}S^5\big)\!\vee\!\big(\bigvee_{i=1}^{c_1} C^5_\eta\big)\!\vee\! C_{\bar{\phi}},\label{eq:SW-C_phi}\\
   \Sigma^2 W&\!\simeq\! P^4(H)\!\vee\!P^6(H)\!\vee\! P^5(T_{\neq 2})\!\vee\!\big(\bigvee_{i=1}^{l-c_1-c_2}S^6\big)\!\vee\!\big(\bigvee_{i=1}^{c_1} C^6_\eta\big)\!\vee\! C_{\Sigma\bar{\phi}},\label{eq:S2W-C_phi}
\end{align}
for some homologically trivial map 
\[\bar{\phi}\colon S^5\to P^3(H_3)\vee\big(\bigvee_{i=1}^{d-c_1}S^3\big)\vee \big(\bigvee_{i=1}^{d}S^4\big)\vee \big(\bigvee_{j=c_2+1}^{t_2}P^4(2^{r_j})\big)\vee \big(\bigvee_{j=1}^{c_2}C^5_{r_j}\big).\]
From now on we assume that $H_3=0$ to study the homotopy type of $\Sigma W$ or the homotopy cofibre $C_{\bar{\phi}}$. 
By Lemma \ref{lem:Moore1} and \ref{lem:Changcpx} we may put 
\begin{equation}\label{map:phi}
\small \bar{\phi}=\sum_{i=1}^{d-c_1}x_i\cdot \eta^2+\sum_{i=1}^{d}y_i\cdot \eta +\sum_{j=c_2+1}^{t_2}(z_j\cdot \tilde{\eta}_{r_j}+\epsilon_j\cdot i_3\eta^2)+\sum_{j=1}^{c_2}w_j\cdot i_P\tilde{\eta}_{r_j}+\theta,
\end{equation}
where all coefficients belong to $\z{}$ and $\theta$ is a linear combination of Whitehead products. By the Hilton-Milnor theorem the domain $\Wh$ of $\theta$ is given by 
\begin{align*}
  \Wh=&\bigoplus_{1\leq i,j\leq d-c_1}\pi_5(\Sigma S^2_i\wedge S^2_j)\oplus \bigoplus_{\substack{1\leq i\leq d-c_1\\c_2+1\leq j\leq t_2}}\pi_5(\Sigma S^2_i\wedge P^3(2^{r_j}))\\
  &\oplus \bigoplus_{\substack{1\leq i\leq d-c_1\\1\leq j\leq c_2}}\pi_5(\Sigma S^2_i\wedge C^4_{r_j})\oplus \bigoplus_{\substack{c_2+1\leq i,j\leq t_2}}\pi_5(\Sigma P^3(2^{r_i})\wedge P^3(2^{r_j}))\\
  &\oplus \bigoplus_{\substack{c_2+1\leq i\leq t_2\\1\leq j\leq c_2}}\pi_5(\Sigma P^3(2^{r_i})\wedge C^4_{r_j})\oplus \bigoplus_{\substack{1\leq i,j\leq c_2}}\pi_5(\Sigma C^4_{r_i}\wedge C^4_{r_j}).  
\end{align*}
Note that all the spaces $\Sigma X_i\wedge X_j$ are $4$-connected and hence there are Hurewicz isomorphisms $\pi_5(\Sigma X_i\wedge X_j)\cong H_5(\Sigma X_i\wedge X_j)$. For different $X_i$ and $X_j$, we use the ambiguous notations 
\[\iota_1\colon \Sigma X_i\to \Sigma X_i\vee \Sigma X_j, \quad \iota_2\colon \Sigma X_j\to \Sigma X_i\vee \Sigma X_j\] 
to denote the natural inclusions.  Then we can write 
\begin{equation}\label{eq:theta}
  \theta=a_{ij}+b_{ij}+c_{ij}+e_{ij}+f_{ij},
\end{equation}
where \begin{align*}
  a_{ij}&\colon S^5\xra{a_{ij}'}\Sigma S^2_i\wedge S^2_j\xra{[\iota_1,\iota_2]}\Sigma S^2_i\vee \Sigma S^2_j,\\ 
  b_{ij}&\colon S^5\xra{b_{ij}'}\Sigma S^2_i\wedge P^3(2^{r_j})\xra{[\iota_1,\iota_2]}\Sigma S^2_i\vee \Sigma P^3(2^{r_j}),\\
  c_{ij}&\colon S^5\xra{c_{ij}'}\Sigma S^2_i\wedge C^4_{r_j}\xra{[\iota_1,\iota_2]}\Sigma S^2_i\vee \Sigma C^4_{r_j},\\
  d_{ij}&\colon S^5\xra{d_{ij}'}\Sigma P^3(2^{r_i})\wedge P^3(2^{r_j})\xra{[\iota_1,\iota_2]}\Sigma P^3(2^{r_i})\vee \Sigma P^3(2^{r_j}),\\
  e_{ij}&\colon S^5\xra{e_{ij}'}\Sigma P^3(2^{r_i})\wedge C^4_{r_j}\xra{[\iota_1,\iota_2]}\Sigma P^3(2^{r_i})\vee \Sigma C^4_{r_j},\\
  f_{ij}&\colon S^5\xra{f_{ij}'}\Sigma C^4_{r_j}\wedge C^4_{r_i}\xra{[\iota_1,\iota_2]}\Sigma C^4_{r_i}\vee \Sigma C^4_{r_j}.
\end{align*} 

Since the homotopy cofibre of $\phi$ is $\Sigma W$, similar arguments to the proof of \cite[Lemma 4.2]{CS22} show the following lemma.
\begin{lemma}\label{lem:cupprods=0}
  Let $C_u$ be the homotopy cofibre of a map $u$ with $u$ given by (1) $u=a_{ij}$, (2) $u=b_{ij}$, (3) $u=c_{ij}$, (4) $u=d_{ij}$, (5) $u=e_{ij}$, (6) $u=f_{ij}$. Then all cup products in $H^\ast(C_u;R)$ are trivial for any principal ideal domain $R$.
\end{lemma}

By Lemma \ref{lem:cupprods=0} and \ref{lem:cupprods=0=nullhtp} we then get 
\begin{corollary}\label{cor:theta=0}
  The Whitehead product component $\theta$ (\ref{eq:theta}) of $\bar{\phi}$ is trivial.
\end{corollary}

For each $n\geq 2$, let $\Theta_n$ be secondary cohomology operation based on the null-homotopy of the composition
\[K_n\xra{\theta_n=\smatwo{\Sq^2\Sq^1}{\Sq^2}} K_{n+3}\times K_{n+2}\xra{\varphi_n=[\Sq^1,\Sq^2]}K_{n+4},\]
where $K_m=K(\z{},m)$ denotes the Eilenberg-MacLane space of type $(\z{},m)$. More concretely,  $\Theta_n\colon S_n(X)\to T_n(X)$ is a cohomology operation with
 \begin{align*}
  S_n(X)&=\ker(\theta_n)_\sharp=\ker(\Sq^2)\cap \ker(\Sq^2\Sq^1)\\
  T_n(X)&=\coker(\Omega\varphi_n)_\sharp=H^{n+3}(X;\z{})/\im(\Sq^1+\Sq^2). 
\end{align*} 
Note that $\Theta_n$ detects the maps $\eta^2\in\pi_{n+2}(S^n)$ and $i_n\eta^2\in\pi_{n+2}(P^{n+1}(2^r))$ (cf. 
\cite[Section 2.4]{lipc23}).  By the method outlined in \cite[page 32]{MM79}, the stable secondary operation $\Theta=\{\Theta_n\}_{n\geq 2}$ is \emph{spin trivial} (cf. \cite{Thomas67}), which means the following Lemma holds.
\begin{lemma}\label{lem:spin-trivial}
The secondary operation $\Theta\colon H^\ast(M;\z{})\to H^{\ast+3}(M;\z{})$ is trivial for any orientable closed smooth spin manifold $M$.
\end{lemma}

Now we are prepared to classify the homotopy types of $C_{\bar{\phi}}$. Note that for a closed orientable smooth $5$-manifold $M$, the second Stiefel-Whitney class equals the second Wu class $v_2$, which satisfies $\Sq^2(x)=v_2\smallsmile x$ for all $x\in H^3(M;\z{})$ \cite[page 132]{MS74}. It follows that the orientable smooth $5$-manifold $M$ is spin if and only if $\Sq^2$ acts trivially on $H^3(M;\z{})$, which is equivalent to $\Sq^2$ acting trivially on $H^4(\Sigma W;\z{})$ or $H^4(C_{\bar{\phi}};\z{})$, by Lemma \ref{lem:ST} and the homotopy decomposition (\ref{eq:SW-C_phi}). 

\begin{proposition}\label{prop:Sq2eq0}
If $M$ is a closed orientable smooth spin $5$-manifold, then there is a homotopy equivalence 
\[C_{\bar{\phi}}\simeq \big(\bigvee_{i=1}^{d-c_1}S^3\big)\!\vee\! \big(\bigvee_{i=1}^{d}S^4\big)\!\vee\! \big(\bigvee_{j=c_2+1}^{t_2}P^4(2^{r_j})\big)\!\vee\!\big(\bigvee_{j=1}^{c_2}C^5_{r_j}\big)\!\vee\! S^6.\]

 \begin{proof}
The smooth spin condition on $M$, together with Lemma \ref{lem:spin-trivial}, implies that $x_i=\epsilon_j=0$ for all $i,j$ in (\ref{map:phi}).
By the comments above Proposition \ref{prop:Sq2eq0}, $M$ is spin implies that the Steenrod square $\Sq^2$ acts trivially on $H^4(C_{\bar{\phi}};\z{})$.   
Then Lemma \ref{lem:Sq2-Chang} and \ref{lem:Sq2} imply $y_i=z_j=w_j=0$ for all $i,j$. Thus the map $\bar{\phi}$ in (\ref{map:phi}) is null-homotopic and therefore we get the homotopy equivalence in the Proposition.
 \end{proof}
\end{proposition}

\begin{remark}\label{rmk:PDcpx}
  If $M$ is a general $5$-dimensional connected \PD complex such that $\Sq^2$ acts trivially on $H^3(M;\z{})$, then we have the following two additional possibilities for the homotopy types of $C_{\bar{\phi}}$ in terms of the secondary cohomology operation $\Theta$:
  \begin{enumerate}
    \item If for any $u\in H^3(M;\z{})$ with $\Theta(u)\neq 0$ and any $v\in \ker(\Theta)$, there holds $\beta_r(u+v)=0$ for all $r$, then there is a homotopy equivalence 
    \[ C_{\bar{\phi}}\simeq \big(\bigvee_{i=2}^{d-c_1}S^3\big)\!\vee\! \big(\bigvee_{i=1}^{d}S^4\big)\!\vee\!  \big(\bigvee_{j=c_2+1}^{t_2}P^4(2^{r_j})\big)\!\vee\! \big(\bigvee_{j=1}^{c_2}C^5_{r_j}\big)\!\vee\! (S^3\cup_{\eta^2}e^6). \]
    \item If there exist $u\in H^3(M;\z{})$ with $\Theta(u)\neq 0$ and $v\in \ker(\Theta)$ such that $\beta_r(u+v)\neq 0$, then there is a homotopy equivalence 
    \[  C_{\bar{\phi}}\simeq \big(\bigvee_{i=1}^{d-c_1}S^3\big)\!\vee\! \big(\bigvee_{i=1}^{d}S^4\big)\!\vee\!  \big(\bigvee_{j_0\neq j=c_2+1}^{t_2}P^4(2^{r_j})\big)\!\vee\! \big(\bigvee_{j=1}^{c_2}C^5_{r_j}\big)\!\vee\! A^6(2^{r_{j_0}}\eta^2), \]
    where $A^6(2^{r_{j_0}}\eta^2)=P^4(2^{r_{j_0}})\cup_{i_3\eta^2}e^6$, $j_0$ is the index such that $r_{j_0}$ is the maximum of $r_j$ satisfying $\beta_{r_j}(u+v)\neq 0$.
  \end{enumerate}

\end{remark}

\begin{proposition}\label{prop:Sq2neq0}
 Suppose that $\Sq^2$ acts non-trivially on $H^3(M;\z{})$, or equivalently $\Sq^2$ acts non-trivially on $H^4(C_{\bar{\phi}};\z{})$. 
\begin{enumerate}
  \item\label{Sq2neq0-1} If for any $u, v\in H^4(C_{\bar{\phi}};\z{})$ satisfying $\Sq^2(u)\neq 0$ and $\Sq^2(v)=0$, there holds $u+v\notin \im(\beta_r)$ for any $r\geq 1$,
  then there is a homotopy equivalence 
  \[C_{\bar{\phi}}\simeq \big(\bigvee_{i=1}^{d-c_1}S^3\big)\!\vee\! \big(\bigvee_{i=2}^{d}S^4\big)\!\vee\! \big(\bigvee_{j=c_2+1}^{t_2}P^4(2^{r_j})\big)\!\vee\!\big(\bigvee_{j=1}^{c_2}C^5_{r_j}\big)\!\vee\! C^6_\eta.\]
 \item\label{Sq2neq0-2} If there exist $u, v\in H^4(C_{\bar{\phi}};\z{})$ with $\Sq^2(u)\neq 0$ and $v\in \ker(\Sq^2)$ such that $u+v\in\im(\beta_r)$ for some $r$, then either there is a homotopy equivalence 
 \[ C_{\bar{\phi}}\simeq \big(\bigvee_{i=1}^{d-c_1}S^3\big)\!\vee\! \big(\bigvee_{i=1}^{d}S^4\big)\!\vee\! \big(\bigvee_{j_1\neq j=c_2+1}^{t_2}P^4(2^{r_j})\big)\!\vee\!\big(\bigvee_{j=1}^{c_2}C^5_{r_j}\big)\!\vee\! A^6(\tilde{\eta}_{r_{j_1}}),\]
or there is a homotopy equivalence 
\[C_{\bar{\phi}}\simeq \big(\bigvee_{i=1}^{d-c_1}S^3\big)\!\vee\! \big(\bigvee_{i=1}^{d}S^4\big)\!\vee\! \big(\bigvee_{j=c_2+1}^{t_2}P^4(2^{r_j})\big)\!\vee\!\big(\bigvee_{j_1\neq j=1}^{c_2}C^5_{r_j}\big)\!\vee\! A^6(i_P\tilde{\eta}_{r_{j_1}}),\]
where the last two complexes are defined by (\ref{eq:An+3}) and $r_{j_1}$ is the minimum of $r_j$ such that $u+v\in \im(\beta_{r_j})$. 
\end{enumerate}

\begin{proof}
  Recall the equation for $\bar{\phi}$ given by (\ref{map:phi}).
 Since $\Sq^2$ acts non-trivially on $H^4(C_{\bar{\phi}};\z{})$, at least one of $y_i,z_j,w_j$ equals $1$. 

(1) The conditions in (\ref{Sq2neq0-1}) implying that $z_j=w_j=0$ for all $j$ and hence $y_i=1$ for some $i$. Clearly we may assume that 
$y_1=1$ and $y_i=0$ for all $2\leq i\leq d$. By the equivalences 
\[\matwo{\eta}{\eta^2}\sim \matwo{\eta}{0}\colon S^5\to S^4\vee S^3,~~\matwo{\eta}{i_3\eta^2}\sim \matwo{\eta}{0}\colon S^5\to S^4\vee P^4(2^r),\]
we may further assume that $x_i=\epsilon_i=0$ for all $i$ in (\ref{map:phi}). Thus we have
\[\bar{\phi}=\eta\colon S^5\to S^4,\]
which proves the homotopy equivalence in (\ref{Sq2neq0-1}).

(2) The conditions in (\ref{Sq2neq0-2}) implies that $z_j=1$ or $w_j=1$ for some $j$. 
For maps $\tilde{\eta}_r,i_3\eta^2\colon S^5\to P^4(2^r)$ and $i_P\tilde{\eta}_s\colon S^5\to C^5_s$, the formulas (\ref{eq:chi}) and (\ref{eq:chi-eta}) indicate the following equivalences 
\begin{align*}
  &\matwo{\tilde{\eta}_r}{\eta^a}\sim \matwo{\tilde{\eta}_r}{0}~(a=1,2),\quad \matwo{i_P\tilde{\eta}_r}{\eta^a}\sim \matwo{i_P\tilde{\eta}_r}{0}~(a=1,2);\\
  &\matwo{\tilde{\eta}_r}{\tilde{\eta}_s}\sim \matwo{\tilde{\eta}_r}{0}~(r\leq s),\quad \matwo{i_P\tilde{\eta}_r}{i_P\tilde{\eta}_s}\sim \matwo{i_P\tilde{\eta}_r}{0} ~(r\leq s);\\
  &\matwo{\tilde{\eta}_r}{i_3\eta^2}\sim \matwo{\tilde{\eta}_r}{0}(i_3\eta^2\in \pi_5(P^4(2^s)), r\neq s),\quad \matwo{i_P\tilde{\eta}_r}{i_3\eta^2}\sim \matwo{i_P\tilde{\eta}_r}{0}.
\end{align*}
It follows that we may assume that $x_i=y_i=0$ for all $i$ regardless of whether $z_j=1$ or $w_j=1$. 
\begin{enumerate}[(i)]
  \item  If $z_j=1$ for some $j$, we assume that $z_j=1$ for exactly one $j$, say $z_{j_1}=1$; in this case, $\epsilon_j=0$ for all $j\neq j_1$. Note that $\id_P+i_3\eta q_4$ is a self-homotopy equivalence of $P^4(2^r)$ and 
  \[(\id_P+i_3\eta q_4)(\tilde{\eta}_r+i_3\eta^2)=\tilde{\eta}_r+i_3\eta^2+i_3\eta^2=\tilde{\eta}_r, \] 
  we may assume that $\epsilon_{j_1}=1$ and $\epsilon_j=0$ for $j\neq j_1$.
  \item If $w_j=1$ for some $j$, then $w_j=1$ for exactly one $j$, say $w_{j_2}=1$; in this case, $\epsilon_j=0$ for all $j$.
\end{enumerate} 

By (\ref{eq:C_r-P}) we have the equivalences for maps $S^5\to P^4(2^r)\vee C^5_s$: 
    \[\matwo{\tilde{\eta}_r}{i_P\tilde{\eta}_s}\sim  \matwo{\tilde{\eta}_r}{0} \text{ if }r\leq s; \quad \matwo{\tilde{\eta}_r}{i_P\tilde{\eta}_s}\sim \matwo{0}{i_P\tilde{\eta}_s}\text{ if }r>s.\] 
Thus we may assume that $\bar{\phi}=\tilde{\eta}_{r_{j_1}}$ if $r_{j_1}\leq r_{j_2}$; otherwise $\bar{\phi}=i_P\tilde{\eta}_{r_{j_2}}$, which prove the homotopy equivalences in (\ref{Sq2neq0-2}).
\end{proof}
\end{proposition}

\begin{proof}[Proof of Theorem \ref{thm:main}]
 Combine Lemma \ref{lem:ST}, the homotopy decomposition (\ref{eq:SW-C_phi}) and Proposition \ref{prop:Sq2eq0} and \ref{prop:Sq2neq0}.
\end{proof}

\begin{proof}[Proof of Theorem \ref{thm:dbsusp}]
 The homotopy types of the discussion of the suspension $\Sigma C_{\bar{\phi}}$ is totally similar to that of $C_{\bar{\phi}}$. The Theorem then follows by Lemma \ref{lem:ST}, the homotopy decomposition (\ref{eq:S2W-C_phi}) and the suspended version of Proposition \ref{prop:Sq2eq0} and \ref{prop:Sq2neq0}.
\end{proof}

\section{Some Applications}\label{sec:appl}
In this section we apply the homotopy decomposition of $\Sigma^2M$ given by Theorem \ref{thm:main} to study the reduced $K$-groups and the cohomotopy sets of $M$. 
\subsection{Reduced $K$-groups}\label{sec:k-grps}
To prove Corollary \ref{cor:kgrps} we recall that the reduced complex $K$-group $\wk(S^n)$ is isomorphic to $\Z$ if $n$ is even, otherwise $\wk(S^n)=0$; the reduced $KO$-groups of spheres are given by
\begin{equation}\label{ko:spheres}
  \begin{tabular}{cccccccccc}
    \toprule 
    $i \pmod 8$&$0$&$1$&$2$&$3$&$4$& $5$&$6$&$7$\\
    \midrule 
    $\ko(S^i)$&$\Z$&$\z{}$&$\z{}$&$0$&$\Z$&$0$&$0$&$0$
   \\ \bottomrule
  \end{tabular}.
\end{equation}

Using the reduced complex $K$-groups and $KO$-groups of spheres one can easily get the following lemma, where the notations $A^7(\tilde{\eta}_r)$ and $A^7(i_P\tilde{\eta}_r)$ refer to (\ref{eq:An+3}).

\begin{lemma}\label{lem:ko-grps}
  Let $m,r$ be positive integers and let $p$ be a prime.
  \begin{enumerate}
    \item $\wk(P^{2m}(p^r))\cong\zp{r}$ and $\wk(P^{2m+1}(p^r))=0$. 
    \item $\wk(C^{2m}_\eta)\cong\Z\oplus\Z$ and $\wk(C^{2m+1}_\eta)=0$.
    \item  $\wk(C^{6}_r)\cong\wk(A^7(i_P\tilde{\eta}_r))\cong\Z$, $\wk(A^7(\tilde{\eta}_r))=0$.
    \item $\ko^2(P^{4+i}(p^r))=\ko^2(C^7_\eta)=0$ for $p\geq 3$ and $i=0,1,2$.
    \item $\ko^2(P^5(2^r))\cong \ko^2(A^7(\tilde{\eta}_r))\cong\z{}$.
    \item $\ko^2(C^6_\eta)\cong\ko^2(C^6_r)\cong \ko^2(A^7(i_P\tilde{\eta}_r))\cong \Z\oplus\z{}$.
  \end{enumerate}

\end{lemma}

\begin{proposition}\label{prop:kgrps}
  Let $M$ be an orientable smooth closed $5$-manifold given by Theorem \ref{thm:main} or \ref{thm:dbsusp}. There hold isomorphisms
  \[\wk(M)\cong\Z^{d+l}\oplus H\oplus H,\quad \ko(M)\cong\Z^l\oplus(\z{})^{l+d+t_2}.\]

  \begin{proof}
    We only give the proof of $\ko(M)$ here, because the proof of $\wk(M)$ is similar but simpler. By Theorem \ref{thm:main} we can write 
    \begin{multline*}
      \Sigma^2 M\simeq \big(\bigvee_{i=1}^lS^3\big)\!\vee\! \big(\bigvee_{i=1}^{d-c_1}S^4\big)\!\vee\! \big(\bigvee_{i=2}^{d}S^5\big)\!\vee\! \big(\bigvee_{i=1}^{l-c_1-c_2}S^6\big)\!\vee\! P^4(H)\!\vee\! P^6(H)\\\!\vee\! \big(\bigvee_{i=1}^{c_1} C^6_\eta\big)\!\vee\! P^5(\frac{T[c_2]}{\z{r_{j_1}}})\!\vee\!\big(\bigvee_{j_2\neq j=1}^{c_2}C^6_{r_j}\big)\!\vee\! \Sigma^2X, 
    \end{multline*}
    where $ \Sigma^2X\simeq \big(S^5\vee P^5(2^{r_{j_1}})\vee C^6_{r_{j_2}}\big)\cup e^7$.
    By Lemma \ref{lem:ko-grps} and the table (\ref{ko:spheres}), there is a chain of isomorphisms 
    \begin{align*}
      \ko(M)\cong&\ko^2(\Sigma^2M)\cong \bigoplus_{l}\ko^2(S^3)\oplus \bigoplus_{d-c_1}\ko^2(S^4)\oplus \bigoplus_{d}\ko^2(S^5)\\
      &\oplus \bigoplus_{l-c_1-c_2}\ko^2(S^6)\oplus \ko(P^4(H)\vee P^6(H))\oplus\bigoplus_{c_1}\ko^2(C^6_\eta)\\
      &\oplus \ko^2(P^5\big(\frac{T[c_2]}{\z{r_{j_1}}}\big))\oplus \bigoplus_{j_2\neq j=1}^{c_2}\ko^2(C^6_{r_j})\oplus \ko^2(\Sigma^2X)\\
      \cong & (\z{})^{l+d-c_1} \oplus \Z^{l-c_1-c_2}\oplus (\Z\oplus\z{})^{\oplus c_1}\oplus (\z{})^{t_2-c_2-1}\\
      &\oplus (\Z\oplus\z{})^{\oplus (c_2-1)}\oplus \ko^2(\Sigma^2X)\\
  \cong & \Z^{l}\oplus (\z{})^{l+d+t_2},
    \end{align*}
  where $\ko^2(\Sigma^2X)\cong\Z\oplus\z{}\oplus\z{}$ in all cases of Theorem \ref{thm:main} can be easily computed by Lemma \ref{lem:ko-grps}.
  \end{proof}  
\end{proposition}

\subsection{Cohomotopy sets}\label{sec:cohtp}
Let $M$ be a closed $5$-manifold.
It is clear that the \emph{cohomotopy Hurewicz maps} 
\[h^i\colon \pi^i(M)\to H^i(M), \quad \alpha\mapsto \alpha^\ast(\iota_i)\] with $\iota_i\in H^i(S^i)$ a generator are isomorphisms for $i=1$ or $i\geq 5$.
For $\pi^4(M)$, there is a short exact sequence of abelian groups (cf. \cite{Steenrod47})
\[0\to \frac{H^5(M;\z{})}{\Sq^2_\Z(H^3(M;\Z))}\to \pi^4(M)\xra{h^4}H^4(M)\to 0,\]
which splits if and only if there holds an equality (cf. \cite[Section 6.1]{Taylor12})
\[\Sq^2_\Z(H^3(M;\Z))=\Sq^2(H^3(M;\z{}))\subseteq H^5(M;\z{}).\]

%
%Given $w\in H^2(M)$, by exactness $h^{-1}(w)\subseteq \pi^2(M)$ is not empty if and only if the composition $\jmath\circ w$ is null-homotopic.

The standard action of $S^3$ on $S^2=S^3/S^1$ by left translation induces a natural action of $\pi^3(M)$ on $\pi^2(M)$. More concretely, the Hopf fibre sequence 
\[S^1\xra{} S^3\xra{\eta}S^2\xra{\imath_2}\CP{\infty}\xra{\jmath}\HP{\infty}\]  induces an exact sequence of sets 
\begin{equation}\label{ES:S2}
  \pi^1(M)\xra{\kappa_u} \pi^3(M)\xra{\eta_\sharp}\pi^2(M)\xra{h}H^2(M)\xra{\jmath_\sharp}\pi^4(M),
\end{equation}
where $[M,\HP{\infty}]=\pi^4(M)$ because  $\HP{\infty}$ has the $6$-skeleton $S^4$, $h=h^2$ is the second cohomotopy Hurewicz map. The homomorphism $\kappa_u$ in (\ref{ES:S2}) is given by the following lemma.

\begin{lemma}[cf. Theorem 3 of \cite{KMT12}]\label{lem:2-cohtp}
  The natural action of $\pi^3(M)$ on $\pi^2(M)$ is transitive on the fibres of $h$ and the stabilizer of $u\in \pi^2(M)$ equals the image of the homomorphism 
\[\kappa_u\colon \pi^1(M)\to \pi^3(M),\quad \kappa_u(v)=\kappa (u\times v)\Delta_M,\]
where $\Delta_M$ is the diagonal map on $M$, $\kappa\colon S^2\times S^1\to S^3$ is the conjugation $(gS^1,t)\mapsto gtg^{-1}$ by setting $S^2=S^3/S^1$.
\end{lemma}
Thus, in a certain sense we only need to determine the third cohomotopy group $\pi^3(M)$. Recall the EHP fibre sequence  (cf. \cite[Corollary 4.4.3]{Neisenbook})
\[\Omega^2 S^4\xra{\Omega H}\Omega^2 S^7\xra{} S^3\xra{E}\Omega S^4\xra{H}\Omega S^7,\]
which induces an exact sequence 
\begin{equation}\label{ES:S3}
  [M,\Omega^2S^4]\xra{(\Omega H)_\sharp}[M,\Omega^2S^7]\xra{}[M,S^3]\xra{E_\sharp}[M,\Omega S^4]\to 0,
\end{equation}
where $0=[M,\Omega S^7]=[\Sigma M,S^7]$ by dimensional reason.

\begin{lemma}\label{lem:S2MS4}
Let $M$ be a $5$-manifold given by Theorem \ref{thm:main}. Then
\begin{enumerate}
  \item $[\Sigma^2M,S^7]\cong\Z\langle q_7 \rangle$, where $q_7$ is the canonical pinch map;
  \item $[\Sigma^2M,S^4]$ contains a direct summand $\Z\langle \nu_4q_7\rangle$, where $\nu_4\colon S^7\to S^4$ is the Hopf map.
\end{enumerate}
  \begin{proof}
By Theorem \ref{thm:main}, there is a homotopy decomposition
\[\Sigma^2 M\simeq U\vee V,\]
where $U$ is a $6$-dimensional complex and $V$ belongs to the set 
\[\mathcal{S}=\{S^7,~C^7_\eta, ~A^7(\tilde{\eta}_{r_{j_1}})=P^5(2^{r_{j_1}})\cup_{\tilde{\eta}_{r_{j_1}}}e^7, ~A^7(i_P\tilde{\eta}_{r_{j_1}})=C^6_{r_{j_1}}\cup_{i_P\tilde{\eta}_{r_{j_1}}}e^7\}.\]
Let $q_V\colon \Sigma^2M\to V$ be the pinch map onto $V$. Then it is clear that the pinch map $q_7$ factors as the composite $\Sigma^2M\xra{q_V} V\xra{q_7 \text{ or } \id_7} S^7$.
We immediately have the chain of isomorphisms 
\[[\Sigma^2M,S^7]\xleftarrow[\cong]{q_V^\sharp} [V,S^7]\cong\Z\langle q_7 \rangle.\]  
For the group $[\Sigma^2M,S^4]$, we show that the direct summand $[V,S^4]$ (through the homomorphism $q_V^{\sharp}$) is isomorphic to $\Z\langle \nu_4q_7\rangle\oplus \Z/12$ for any $V\in \mathcal{S}$.

If $V=S^7$, we clearly have $[S^7,S^4]\cong \Z\langle \nu_4 \rangle\oplus\Z/12$. If $V=C^7_\eta$, then from the homotopy cofibre sequence
 \[S^6\xra{\eta}S^5\xra{i_5}C^7_\eta\xra{q_7}S^7\xra{\eta}S^6\]
we have an exact sequence 
 \[0\to\pi_7(S^4)\xra{q_7^\sharp}[C^7_\eta,S^4]\xra{i_5^\sharp}\pi_5(S^4)\xra{\eta^\sharp}\pi_6(S^4).\]
 Since $\eta^\sharp$ is an isomorphism, $i_5^\sharp$ is trivial and hence $q_7^\sharp$ is an isomorphism. Thus we have \[ [C^7_\eta,S^4]\cong (q_7)^\sharp(\pi_7(S^4))\cong \Z\langle \nu_4q_7\rangle\oplus\Z/12.\]
 If $V=A^7(\tilde{\eta}_r)= P^5(2^{r_{j_1}})\cup_{\tilde{\eta}_{r_{j_1}}}e^7$, the homotopy cofibre sequence 
 \[S^6\xra{\tilde{\eta}_{r_{j_1}}}P^5(2^{r_{j_1}})\xra{i_P}A^7(\tilde{\eta}_r)\xra{q_7}S^7\xra{}P^6(2^{r_{j_1}})\]
implying an exact sequence 
\[0\to \pi_7(S^4)\xra{q_7^\sharp}[A^7(\tilde{\eta}_r),S^4]\xra{i_P^\sharp}[P^5(2^{r_{j_1}}),S^4]\xra{\tilde{\eta}_{r_{j_1}}^\sharp}\pi_6(S^4). \]
Since $[P^5(2^{r_{j_1}}),S^4]\cong\z{}\langle  \eta q_5\rangle$, the formula $q_5\tilde{\eta}_{r_{j_1}}=\eta$ in (\ref{eq:chi-eta}) then implying 
 $\tilde{\eta}_{r_{j_1}}^\sharp$ is an isomorphism. Thus 
 \[ [A^7(\tilde{\eta}_r),S^4]\cong (q_7)^\sharp(\pi_7(S^4))\cong \Z\langle \nu_4q_7\rangle\oplus\Z/12.\]
 The computations for $V=A^7(i_P\tilde{\eta}_r)$ is similar. Firstly it is clear that  
 \[[C^6_{r_{j_1}},S^4]\xleftarrow[\cong]{i_P^\sharp} [P^5(2^{r_{j_1}}),S^4]\cong\z{}\langle  \eta q_5\rangle.\]
 Recall we have the composite $q_5\colon P^5(2^{r_{j_1}}) \xra{i_P}C^6_{r_{j_1}}\xra{q_5}S^5$. 
It follows that the homomorphism $ [C^6_{r_{j_1}},S^4]\xra{(i_P \tilde{\eta}_{r_{j_1}})^\sharp} \pi_6(S^4)$ is an isomorphism, and thus there is an isomorphism 
\[[A^7(i_P\tilde{\eta}_r),S^4]\cong (q_7)^\sharp(\pi_7(S^4))\cong \Z\langle \nu_4q_7\rangle\oplus\Z/12. \]  
\end{proof}
\end{lemma}

\begin{lemma}\label{lem:A6S4}
  Let $r\geq 1$ be an integer. There hold isomorphisms 
  \begin{enumerate}
    \item $[C^5_\eta,S^4]=0$ and $[C^5_r,S^4]\cong \z{r+1}$.
    \item $[A^6(\tilde{\eta}_r),S^4]\cong\z{r-1}$, where $\Z/1=0$ for $r=1$.
    \item $[A^6(i_P\tilde{\eta}_r),S^4]\cong\z{r}$.
  \end{enumerate}
  \begin{proof}
 (1)  The groups in (1) refer to \cite{Baues85} or \cite{lipc2-an2}.

 (2) The homotopy cofibre sequence for $A^6(\tilde{\eta}_r)$, as given in the proof of Lemma \ref{lem:S2MS4}, implying an exact sequence 
 \[ [P^5(2^r),S^4]\xra[\cong]{\tilde{\eta}_r^\sharp}[S^6,S^4]\to [A^6(\tilde{\eta}_r),S^4]\xra{i_P^\sharp} [P^4(2^r),S^4]\xra{\tilde{\eta}_r^\sharp}[S^5,S^4].\]
Thus $(i_P)^\sharp$ is a monomorphism and $\im (i_P)^\sharp=\ker(\tilde{\eta}_r^\sharp)\cong\z{r-1}\langle 2q_4\rangle$. 

(3) The computation of the group $[A^6(i_P\tilde{\eta}_r),S^4]$ is similar, by noting the isomorphism $[C^5_r,S^4]\cong \z{r+1}\langle q_4\rangle$ (cf. \cite{Baues85}).
  \end{proof}
\end{lemma}

\begin{proposition}\label{prop:MS3}
  Let $M$ be a $5$-manifold given by Theorem \ref{thm:main} or \ref{thm:dbsusp}. The homomorphism $(\Omega H)_\sharp$ in (\ref{ES:S3}) is surjective and hence there is an isomorphism
  \[\Sigma \colon \pi^3(M)\to \pi^4(\Sigma M).\]

  Moreover, let $M$ be the $5$-manifold,  together with the integers $c_1,c_2$ and $r_{j_1}$, given by Theorem \ref{thm:main}, then we have the following concrete results:
   \begin{enumerate}
    \item if $M$ is spin, then 
    \[\pi^3(M)\cong \Z^d\oplus (\z{})^{l+1-c_1-c_2}\oplus T[c_2]\oplus (\bigoplus_{j=1}^{c_2}\z{r_{j}+1});\]
    \item if $M$ is non-spin and the conditions in (\ref{nonspin:a}) hold, then 
    \[\pi^3(M)\cong 
     \Z^{d}\oplus (\z{})^{l-c_1-c_2}\oplus T[c_2]\oplus (\bigoplus_{j=1}^{c_2}\z{r_{j}+1});\]
     \item  if $M$ is non-spin and the conditions in (\ref{nonspin:b}) hold, then $\pi^3(M)$ is isomorphic to one of the following groups:
     \begin{align*}
     ~~(i) ~~&  \Z^{d}\oplus (\z{})^{l-c_1-c_2}\oplus \frac{T[c_2]}{\z{r_{j_1}}}\oplus (\bigoplus_{j=1}^{c_2}\z{r_{j}+1})\oplus \z{r_{j_1}-1},\\
     ~~(ii)~~ &\Z^{d}\oplus (\z{})^{l-c_1-c_2}\oplus T[c_2]\oplus (\bigoplus_{j_1\neq j=1}^{c_2}\z{r_{j}+1})\oplus \z{r_{j_1}}.
     \end{align*} 
     
   \end{enumerate}
 
  \begin{proof}
   We firstly apply the exact sequence (\ref{ES:S3}) to show that the suspension $\pi^3(M)\xra{\Sigma}\pi^4(\Sigma M)$ is an isomorphism. By duality, it suffices to show the second James-Hopf invariant $H$ induces a surjection $H_\sharp\colon [\Sigma^2M,S^4]\to [\Sigma^2M,S^7].$
  By Lemma \ref{lem:S2MS4}, there hold isomorphisms
   \[[\Sigma^2M,S^7]\cong \Z\langle q_7\rangle \text{ and } [\Sigma^2M,S^4]\cong\Z\langle \nu_4q_7\rangle\oplus G\]
 for some abelian group $G$. Then the surjectivity of $H_\sharp$ follows by the homotopy equalities
  \[H(\nu_4)=\id_{7},~~H(\nu_4q_7)=H(\nu_4)q_7=q_7.\] 
Note the first statement only depends the homotopy type of the double suspension $\Sigma^2M$, so we can also assume that $M$ is the $5$-manifold satisfying conditions in Theorem \ref{thm:main}. 

The computations of the group $[\Sigma M,S^4]$ follows by Theorem \ref{thm:main}, Lemma \ref{lem:A6S4}:
 \begin{enumerate}
  \item If $M$ is spin, then 
  \begin{multline*}
    [\Sigma M,S^4]\cong \big(\bigoplus_{i=1}^d[S^4,S^4]\big)\oplus \big(\bigoplus_{i=1}^{l-c_1-c_2}[S^5,S^4]\big)\oplus [P^4(T[c_2]),S^4]\\
    \oplus \big(\bigoplus_{j=1}^{c_2}[C^5_{r_{j}},S^4]\big)\oplus [S^6,S^4].
  \end{multline*}
  \item If $M$ is non-spin and $\Sigma M$ is given by (\ref{nonspin:a}), then 
    \begin{multline*}
      [\Sigma M,S^4]\cong \big(\bigoplus_{i=2}^d[S^4,S^4]\big)\oplus \big(\bigoplus_{i=1}^{l-c_1-c_2}[S^5,S^4]\big)\oplus [P^4(T[c_2]),S^4]\\
      \oplus \big(\bigoplus_{j=1}^{c_2}[C^5_{r_{j}},S^4]\big)\oplus [C^6_\eta,S^4]. 
    \end{multline*}
  \item  If $M$ is non-spin and $\Sigma M$ is given by (\ref{nonspin:b}), then 
  \begin{multline*}
    [\Sigma M,S^4]\cong \big(\bigoplus_{i=1}^d[S^4,S^4]\big)\oplus \big(\bigoplus_{i=1}^{l-c_1-c_2}[S^5,S^4]\big)\oplus [P^4\big(\frac{T[c_2]}{\z{r_{j_1}}}\big),S^4]\\
    \oplus \big(\bigoplus_{j=1}^{c_2}[C^5_{r_{j}},S^4]\big)\oplus [A^6(\tilde{\eta}_{r_{j_1}}),S^4],
  \end{multline*} 
or 
\begin{multline*}
  [\Sigma M,S^4]\cong \big(\bigoplus_{i=1}^d[S^4,S^4]\big)\oplus \big(\bigoplus_{i=1}^{l-c_1-c_2}[S^5,S^4]\big)\oplus [P^4(T[c_2]),S^4]\\
  \oplus \big(\bigoplus_{j_1\neq j=1}^{c_2}[C^5_{r_{j}},S^4]\big)\oplus [A^6(i_P\tilde{\eta}_{r_{j_1}}),S^4].
\end{multline*}

 \end{enumerate}

  \end{proof}
\end{proposition}

\bibliographystyle{amsalpha}
\bibliography{refs}

\end{document}